\documentclass[11pt,a4paper,english,reqno]{amsart}
\usepackage[a4paper,footskip=1.5em,bottom=38mm,margin=2.34cm]{geometry}
\usepackage{amsmath,amssymb,amsthm,mathtools}
\usepackage[none]{hyphenat}
\usepackage{comment}
\usepackage{tikz}
\usetikzlibrary{math,calc,intersections}
\usetikzlibrary{positioning,arrows,shapes,decorations.markings,decorations.pathreplacing,matrix,patterns}
\tikzstyle{vertex}=[circle,draw=black,fill=black,inner sep=0,minimum size=3pt,text=white,font=\footnotesize]
\usepackage{graphicx}
\usepackage{hyperref}
\hypersetup{
	colorlinks=true,
	linkcolor=blue,
	filecolor=magenta,
	urlcolor=cyan,
	citecolor=blue
}

\newtheorem{theorem}{Theorem}[section]
\newtheorem{conjecture}[theorem]{Conjecture}
\newtheorem{lemma}[theorem]{Lemma}
\newtheorem{claim}[theorem]{Claim}

\newtheorem{definition}{Definition}
\usepackage{amssymb,amsmath}

\title{The extremal number of surfaces}
\author{Andrey Kupavskii}
\address{G-SCOP, CNRS, University Grenoble-Alpes, France and Moscow Institute of Physics and Technology, Russia}
\email{kupavskii@ya.ru}
\author{Alexandr Polyanskii}
\address{MIPT Moscow}
\email{alexander.polyanskii@yandex.ru}
\author{Istv\'an Tomon}
\address{ETH Zurich}
\email{istvan.tomon@math.ethz.ch}
\author{Dmitriy Zakharov}
\address{MIPT Moscow \and HSE Moscow}
\email{zakharov2k@gmail.com}

\date{}

\begin{document}

\sloppy

\maketitle

\begin{abstract}
In 1973, Brown, Erd\H{o}s and S\'os proved that if $\mathcal{H}$ is a 3-uniform hypergraph on $n$ vertices which contains no triangulation of the sphere, then $\mathcal{H}$ has at most $O(n^{5/2})$ edges, and this bound is the best possible up to a constant factor. Resolving a conjecture of Linial, also reiterated by Keevash, Long, Narayanan, and Scott, we show that the same result holds for triangulations of the torus. Furthermore, we extend our result to every closed orientable surface $\mathcal{S}$.
\end{abstract}

\section{Introduction}
Let $\mathcal{F}$ be a (possibly infinite) family of $r$-uniform hypergraphs. The \emph{Tur\'an number} or \emph{extremal number} of $\mathcal{F}$, denoted by $\mbox{ex}(n,\mathcal{F})$, is the maximum number of edges in an $r$-uniform hypergraph on $n$ vertices which does not contain any member of $\mathcal{F}$ as a subhypergraph. The study of extremal numbers of graphs and hypergraphs is one of the central topics in discrete mathematics, which	goes back more than hundred years to the works of Mantel \cite{Ma} in 1907 and Tur\'an \cite{Tu} in 1941. For a general reference, we refer the reader to the surveys \cite{K11,MPS}. 

In this paper, we are interested in the extremal number of families which arise from topology. An $r$-uniform hypergraph $\mathcal{S}$ naturally corresponds to the simplicial complex formed by the subsets of the edges of $\mathcal{S}$. Therefore, we can talk about \emph{homeomorphisms} between hypergraphs, meaning the homeomorphisms of the corresponding simplicial complexes. Linial \cite{L08,L18}, as a part of the `high-dimensional combinatorics' programme, proposed to study the following extremal question. Given an $r$-uniform hypergraph $\mathcal{S}$, at most how many edges can an $r$-uniform hypergraph on $n$ vertices have if it does not contain a homeomorphic copy of $\mathcal{S}$? Let us denote this number by $\mbox{ex}_{hom}(n,\mathcal{S})$. In the case of $r=2$, a celebrated result of Mader \cite{M67} tells us that if the graph $G$ with $n$ vertices avoids a subdivision of the complete graph $K_t$, then $G$ has at most $cn$ edges, where $c=c(t)$ depends on $t$ only. Since every subdivision of a graph $H$ is homeomorphic to $H$, we deduce that $\mbox{ex}_{hom}(n,H)$ is linear for any fixed graph $H$. 

In this paper, we consider the case $r=3$. Whenever appropriate, we shall talk about triangulations of surfaces rather than homeomorphisms of hypergraphs. An old result of Brown, Erd\H{o}s and S\'os \cite{BES73} states that if a $3$-uniform hypergraph on $n$ vertices does not contain a triangulation of the sphere, then $\mathcal{H}$ has $O(n^{5/2})$ edges, and this bound is the best possible up to the constant factor. Inspired by this result, it is natural to study the related problem for the torus and other surfaces as well. Indeed, Linial \cite{L08,L18} proposed the conjecture that if a  $3$-uniform hypergraph on $n$ vertices does not contain a triangulation of the torus, then $\mathcal{H}$ has at most $O(n^{5/2})$ edges. He proved (unpublished, see \cite{L18} for an outline of the approach he suggests) that if this is true, then this bound is best possible up to a constant factor, and together with Friedgut \cite{L08} they proved the upper bound $O(n^{3-1/3})$. Our main result is the resolution of this conjecture.

\begin{theorem}\label{thm:torus}
There exists a constant $c>0$ such that if $\mathcal{H}$ is a 3-uniform hypergraph on $n$ vertices which does not contain a triangulation of the torus, then $\mathcal{H}$ has at most $cn^{5/2}$ edges.
\end{theorem}

 We can also generalize our result to every closed orientable surface. A surface is \emph{closed} if it is compact and without boundary. By the Classification theorem of closed surfaces (see Theorem \ref{thm:class}), every closed orientable surface is homeomorphic a sphere with $g$ handles for some $g\geq 1$.

\begin{theorem}\label{thm:orientable}
Let $\mathcal{S}$ be a closed orientable surface. There exists $c=c(\mathcal{S})>0$ such that if $\mathcal{H}$ is a 3-uniform hypergraph on $n$ vertices which does not contain a triangulation of $\mathcal{S}$, then $\mathcal{H}$ has at most $cn^{5/2}$ edges.
\end{theorem}

 For completeness, we will present a proof for matching lower bounds as well. Our argument is based on the aforementioned approach of Nati Linial \cite{L18}.

\begin{theorem}\label{thm:lower_bound}
Let $\mathcal S$ be a closed surface. Then there exists $c=c(\mathcal{S})>0$ such that the following holds. For every positive integer $n\geq 3$, there exists a $3$-uniform hypergraph on $n$ vertices with at least $c n^{5/2}$ edges that does not contain a triangulation of $\mathcal S$.  
\end{theorem}

In general, Keevash, Long, Narayanan and Scott \cite{KLNS20} proved that for every 3-uniform hypergraph $\mathcal{S}$ there exists $c=c(\mathcal{S})$ such that $\mbox{ex}_{hom}(n,\mathcal{S})\leq cn^{3-1/5}$. It remains open whether the exponent can be pushed down to $5/2$. If true, this would imply our main results.

Our paper is organized as follows. In the next subsections, we introduce our notation and give a brief outline of the proof of Theorem~\ref{thm:torus}. In Section~\ref{sect:lower_bound} we prove Theorem~\ref{thm:lower_bound}; in Section \ref{sect:torus}, we prove Theorem \ref{thm:torus}; in Section \ref{sect:surface}, we prove Theorem \ref{thm:orientable}.

\subsection{Notation}
 As usual, $[n]$ denotes the set $\{1,\dots,n\}$, and if $X$ is a set, $X^{(r)}$ is the family of $r$-element subsets of $X$. If $G$ is a graph and $X\subset V(G)$, then $N(X)$ denotes the \emph{neighborhood} of $X$, that is, the set of vertices $v\in V(G)\setminus X$ adjacent to at least one element of $X$. 

Let $\mathcal{H}$ be a 3-uniform hypergraph. If $x,y,z\in V(\mathcal{H})$, we write $xyz$ instead of $\{x,y,z\}$ (similarly for graphs as well). If $x\in V(\mathcal{H})$, then $\mathcal{H}_{x}$ is the \emph{link-graph} of $x$. If $x,y\in V(\mathcal{H})$, then $N(x,y)$ is the set of vertices $z$ such that $xyz\in E(\mathcal{H})$. Two edges of $\mathcal{H}$ are \emph{neighboring} if they intersect in two vertices. If $G$ and $H$ are graphs on the same vertex set, then $G\cap H$ is the graph on $V(G)$ with edge set $E(G)\cap E(H)$.

We omit the use of floors and ceilings whenever they are not crucial.

\subsection{Outline of the proof}

In this subsection, we briefly outline the proof of the upper bounds. Let us first summarize the argument of Brown, Erd\H{o}s and S\'os \cite{BES73} for finding a triangulation of the sphere. They show that if a 3-uniform hypergraph $\mathcal{H}$ has $n$ vertices and $\Omega(n^{5/2})$ edges, then $\mathcal{H}$ contains a \emph{double pyramid}. 

\begin{definition}
A \emph{double pyramid} is a 3-uniform hypergraph on  $s+2$ vertices $x,x',y_{1},\dots,y_{s}$ for some integer $s\geq 3$, whose edges are $xy_{i}y_{i+1}$ and $x'y_{i}y_{i+1}$ for $i=1,\dots,s$ (indices are taken modulo $s$). The vertices $x$ and $x'$ are called the \emph{apexes} of the double pyramid.
\end{definition}

Indeed, by a simple averaging argument, one can find two vertices $x,x'$ such that the graph $G:=\mathcal{H}_{x}\cap \mathcal{H}_{x'}$ has at least $n$ edges. But then $G$ contains a cycle $y_{1},\dots,y_{s}$, and $x,x',y_{1},\dots,y_{s}$ forms a double pyramid.

Our approach to construct a triangulation of the torus in $\mathcal{H}$ goes by gluing double pyramids together in a cyclic fashion. In order to do this, we show that if $\mathcal{H}$ has $\Omega(n^{5/2})$ edges, then double pyramids are ``all over the place''. More precisely, we show that almost all pairs of neighboring edges $(e,f)$ in $\mathcal{H}$ have the property that even after deleting a large proportion of the vertices in $V(\mathcal{H})\setminus (e\cup f)$ randomly, we can find a certain double pyramid-lire triangulation of the sphere containing $e$ and $f$ with high probability.  Then, we find a sequence of neighboring edges forming a ``cycle'' (for the precise notion of a cycle, see Section \ref{section:topology}), and for any pair of neighboring edges $(e,f)$ in the cycle, we find a sphere containing $e$ and $f$. The `supersaturation' property for triangulations of a sphere described above allows us to choose these spheres to be pairwise disjoint outside the cycle. The union of these spheres contains a triangulation of the torus.

Now let us outline the proof of Theorem \ref{thm:orientable}. Let $\mathcal{S}$ be an orientable surface of genus $g$, then $\mathcal{S}$ is homeomorphic to a sphere with $g$ handles, where each handle can be thought of as a torus glued to the sphere. We show that in general, gluing hypergraphs along an edge does not increase their extremal number by much, which then implies our result.

\section{Lower bound for the extremal numbers of surfaces}\label{sect:lower_bound}

In this section, we present the proof of Theorem \ref{thm:lower_bound}. We remark that the lower bound construction of Brown, Erd\H{o}s and S\'os \cite{BES73} does not work in our case. Indeed, they used the observation that every triangulation of the sphere contains a hypergraph $T$ consisting of five vertices $x,y_1,y_2,y_3,y_4$ and four edges $xy_{i}y_{i+1}$ for $i=1,2,3,4$ (indices are modulo 4), and they proved that the extremal number of $T$ is already $\Omega(n^{5/2})$. Unfortunately, e.g. the torus has triangulations without $T$. Our proof of the lower bound is based on the probabilistic deletion argument.

\begin{proof}[Proof of Theorem \ref{thm:lower_bound}] Fix a triangulation of $\mathcal S$, and let $v,e,f$ stand for the number of vertices, edges, and faces in this triangulation. By the Classification theorem of closed surfaces (see Theorem \ref{thm:class}), there exists $g\ge 0$ (depending on $\mathcal S$ only) such that $v-e+f = 2-g.$ Moreover, we clearly have $2e = 3f,$ and thus $f = 2v-4+2g$ for any triangulation of $\mathcal S.$

Gao \cite{Gao} has determined the asymptotics (as $v\to \infty$)  for the number of unlabelled rooted triangulations of $\mathcal S$ with $v$ vertices. The formula implies that the number $N_v$ of unlabelled triangulations of $\mathcal S$  is at most $C^v$ for some fixed constant $C>1$ (depending on $g$) and any $v$.

We are now ready to prove the theorem. Consider a random hypergraph $\mathcal H$ on $n$ vertices, including each edge in $\mathcal H$ independently and with probability $p = c_0 n^{-1/2}$, where $c_0>0$ is a sufficiently small constant to be determined later. Let us count the expected number $\eta$ of triangulations of $\mathcal S$ in $\mathcal H.$ We have

$$\mathbb{E}\eta = \sum_{v=4}^n \binom{n}{v} v! N_vp^{2v-4+g}\le \sum_{v=4}^nn^vC^v c_0^{2v-4+g}n^{-\frac 12(2v-4+g)}<n^2/2,$$
provided $c_0<(2C)^{-1}.$ Thus, there exists a choice of $\mathcal H$ with at least $\frac {c_0}{2}n^{5/2}$ edges and with at most $n^2$ different triangulations of $\mathcal S.$ Delete one edge from each such triangulation, obtaining a hypergraph $\mathcal H'$ with at least $\frac {c_0}{2} n^{5/2}-n^2 = \Omega(n^{5/2})$ edges and with no triangulation of $\mathcal S.$   
\end{proof}

\section{Torus --- Upper bound}\label{sect:torus}

In this section, we prove Theorem \ref{thm:torus}.

\subsection{Admissible edges}\label{sect:admissibleedges}

Let $p,\epsilon\in (0,1]$ and let $k$ be a positive integer. In what comes, we define the notion of a \emph{$(p,\epsilon,k)$-admissible edge} in a graph $G$. The definition might seem quite convoluted at first, but in exchange it will be convenient to work with it later.

\begin{definition}
Let $e=xy$ be an edge of $G$. Select each vertex of $G$ independently with probability $p$, and let $U$ be the set of selected vertices. Let $A_{e}$ be the event that there are at least $k$ internally vertex disjoint paths with endpoints $x$ and $y$ in $G[U]$ (not counting the length one path $xy$), conditioned on the event that $x,y\in U$. Then $e$ is \emph{$(p,\epsilon,k)$-admissible} if $\mathbb{P}(A_{e})\geq 1-\epsilon$.
\end{definition}

This section is devoted to the proof of the following lemma.

\begin{lemma}\label{lemma:admissible}
Let $p,\epsilon\in (0,1]$, and let $k,n$ be positive integers. If $G$ is a graph on $n$ vertices, then all but at most $\frac{2k}{p^2\epsilon}n$ edges of $G$ are $(p,\epsilon,k)$-admissible.
\end{lemma}

In the proof of this lemma, we will use the following well known theorem of Mader \cite{M72}, which tells us that every graph of large average degree contains a graph of high connectivity.

\begin{claim}\label{claim:connected}[Mader's theorem]
	If $G$ is a graph with average degree at least $4k$, then $G$ contains a ${(k+1)}$-vertex-connected subgraph.
\end{claim}

\begin{proof}[Proof of Lemma \ref{lemma:admissible}]
Say that an edge $e\in E(G)$ is \emph{bad} if it is not $(p,\epsilon,k)$-admissible, and let $N$ be the number of bad edges. Let $U$ be a subset of $V(G)$ we get by selecting each vertex of $G$ independently with probability $p$. For every edge $e$, let $B_{e}$ be the event that both endpoints of $e$ are in $U$, and there are no $k$ internally vertex disjoint paths (other than the length one path) connecting the endpoints of $e$ in $G[U]$. If $e$ is bad, then $\mathbb{P}(B_{e}|e\subset U)\geq \epsilon$, so $\mathbb{P}(B_{e})\geq \epsilon p^{2}$. Let $X=\sum_{e\in E(G)}I(B_{e})$, where $I(B_e)$ is the indicator random variable of $B_{e}$. Then $X=|F|$, where $F$ is the set of edges in $G[U]$ for which there are no $k$ internally vertex disjoint paths connecting its endpoints in $G[U]$. We have $\mathbb{E}(X)=\sum_{e\in E(G)}\mathbb{P}(B_{e})\geq \epsilon p^{2}N$, so there exists a choice for $U$ for which $X\geq \epsilon p^{2}N$. Consider the subgraph $H$ of $G[U]$ with edge set $F$, then there are no $k$ internally vertex disjoint paths connecting the endpoints of any edge in $H$. By Menger's theorem \cite{M27}, $H$ cannot contain a $(k+1)$-vertex-connected subgraph. By Mader's theorem, this implies that $|F|=|E(H)|<2kn$. Therefore, we get $\epsilon p^{2}N<|F|<2kn$, which gives $N<\frac{2k}{\epsilon p^{2}}n$.
\end{proof}

\subsection{Admissible pairs of hyperedges}

Let $\mathcal{H}$ be a 3-uniform hypergraph. Say that a neighboring pair of edges $(e,f)$, where $e=xyz$ and $f=x'yz$, is $(p,\epsilon,k)$-admissible, if the edge $yz$ in the graph $\mathcal{H}_{x}\cap \mathcal{H}_{x'}$ is $(p,\epsilon,k)$-admissible. Also, for a positive integer $r$, say that $(e,f)$ is $(p,\epsilon,k,r)$-semi-admissible, if there exist at least $r$ edges $g=x''yz$ of $\mathcal{H}$ such that $(e,g)$ and $(g,f)$ are both $(p,\epsilon,k)$-admissible.

In this section, we prove the following lemma.

\begin{lemma}\label{lemma:3admissible}
Let $\epsilon,p\in[0,1]$, and let $k,r,n$ be positive integers. Let $\mathcal{H}$ be a 3-uniform hypergraph with $n$ vertices and at least $\frac{12r}{p}\sqrt{\frac{k}{\epsilon}}n^{5/2}$ edges. Then $E(\mathcal{H})$ contains a subset $F$ of at least $\frac{1}{2}|E(\mathcal{H})|$ edges such that any pair of neighboring edges in $F$ is $(p,\epsilon,k,r)$-semi-admissible in $\mathcal{H}$.
\end{lemma}

Let us prepare the proof with a simple claim.

\begin{claim}\label{claim:2path}
Let $r$ be a positive integer. Let $G$ be a graph on $n$ vertices in which the number of non-edges is at most $q$. Then $V(G)$ contains a set $W$ of at least $n-\frac{2rq}{n}$ vertices such that any pair of vertices in $W$ is joined by at least $r$ paths of length two in $G$. 
\end{claim}

\begin{proof}
Let $v_{1},\dots,v_{r}\in V(G)$ be vertices in $G$ with the $r$ largest degrees, and let $W$ be the common neighborhood of $v_{1},\dots,v_{r}$. Then $n-|W|\leq \sum_{i=1}^{r}\mbox{deg}_{\overline{G}}(v_{i})\leq \frac{2rq}{n}.$ 
\end{proof}

\begin{proof}[Proof of Lemma \ref{lemma:3admissible}]
Let  $\{y,z\}$ be a pair of vertices in $V(\mathcal{H})$, and let $n_{y,z}=|N(y,z)|$.  Let $H_{y,z}$ be the graph on $N(y,z)$ in which $x$ and $x'$ are joined by an edge if the pair $(xyz,x'yz)$ is \emph{not} $(p,\epsilon,k)$-admissible. Also, let $H'_{y,z}$ be the graph on $N(y,z)$ in which $x$ and $x'$ are joined by an edge if $(xyz,x'yz)$ is \emph{not} $(p,\epsilon,k,r)$-semi-admissible.

Note that by Claim \ref{claim:2path}, one can delete a set $R_{y,z}$ of at most $2r|E(H_{y,z})|/n_{y,z}$ vertices of $H'_{y,z}$ to make  it an empty graph. Let $T_{y,z}=\{xyz:x\in R_{y,z}\}$ be the set of edges of $\mathcal{H}$ corresponding to the elements of $R_{y,z}$, and let $T=\bigcup_{y,z\in V(\mathcal{H})}T_{y,z}$. Then $F=E(\mathcal{H})\setminus T$ has the desired property. It remains to show that $|F|\geq \frac{1}{2}|E(\mathcal{H})|$, which will follow from the inequality $|T|\leq \frac{6r}{p}\sqrt{\frac{k}{\epsilon}}n^{5/2}$. For simplicity, write $\alpha=\frac{1}{p}\sqrt{\frac{k}{\epsilon}}$.

We have 
$$|T|\leq \sum_{y,z\in V(\mathcal{H}}|T_{y,z}|\leq 2r\sum_{y,z\in V(\mathcal{H})}\frac{|E(H_{y,z})|}{n_{y,z}}.$$
Note that if $n_{y,z}\leq \alpha n^{1/2}$, then $\frac{|E(H_{y,z})|}{n_{y,z}}<\alpha n^{1/2}$, so the contribution of such terms to the sum is at most $\alpha n^{5/2}$. Therefore, we have
$$|T|<2r\alpha n^{5/2}+\frac{2r}{\alpha}n^{-1/2}\sum_{y,z\in V(\mathcal{H})}|E(H_{y,z})|.$$

For a pair of vertices $\{x,x'\}$, let $G_{x,x'}$ be the subgraph of $\mathcal{H}_{x}\cap \mathcal{H}_{x'}$ formed by the not $(p,\epsilon,k)$-admissible edges of $\mathcal{H}_{x}\cap \mathcal{H}_{x'}$. Then $|E(G_{x,x'})|\leq \frac{2k}{p^{2}\epsilon}n=2\alpha^2 n$ by Lemma \ref{lemma:admissible}. Also, note that each edge of $G_{x,x'}$ corresponds to a pair of neighboring edges $(e,f)$ that is not $(p,\epsilon,k)$-admissible, so we have $$\sum_{y,z\in V(\mathcal{H})}|E(H_{y,z})|=\sum_{x,x'\in V(\mathcal{H})}|E(G_{x,x'})|\leq 2\alpha^{2}n^{3}.$$
But then $|T|\leq 6r\alpha n^{5/2},$ finishing the proof.
\end{proof}

\subsection{Topology}\label{section:topology}

This section contains all the topological notions and results we are going to use.  This section should be comprehensible with basic knowledge of topology.

Let us start with the classification theorem of closed surfaces. Recall that the Euler characteristic of a closed surface $\mathcal S$ is equal to $v-e+f$, where $v,e,f$ are the numbers of vertices, edges and faces, respectively, in a triangulation of $\mathcal S$. The Euler characteristic does not depend on the triangulation.  
\begin{theorem}[Classification theorem of closed surfaces]\label{thm:class}
Any connected closed surface is homeomorphic to a member of one of the following families:
\begin{enumerate}
    \item sphere with $g$ handles for $g\geq 0$ (which has Euler characteristic $2-2g$),
    \item sphere with $k$ cross-caps for $k\geq 1$ (which has Euler characteristic $2-k$). 
\end{enumerate}
The surfaces in the first family are the orientable surfaces.
\end{theorem}
 
Now let us introduce our notion of cycle, which we will use to glue spheres along to get a triangulation of the torus.

\begin{definition}
A 3-uniform hypergraph $\mathcal{C}$ is called a \emph{topological cycle} if $\mathcal{C}$ is a triangulation of the cylinder $S^1\times [0,1]$, or the M\"obius strip such that each vertex of $\mathcal C$ lies on the boundary.
\end{definition}

 If $\mathcal{C}$ is a topological cycle, then its edges can be ordered cyclically such that consecutive edges are neighboring, that is, they share two common vertices. Call such an ordering \emph{proper} for $\mathcal{C}$. Note that a topological cycle with $r$ edges has $r$ vertices, and the link graph of every vertex is a path. Also, the Euler characteristic of a topological cycle is 0.

For example, tight cycles are topological cycles. For $r\geq 4$, a \emph{tight cycle of length $r$} is the 3-uniform hypergraph on vertices $x_{1},\dots,x_{r}$ with edges $x_{i-1}x_{i}x_{i+1}$ for $i=1,\dots,r$, where indices are meant modulo $r$. Indeed, a tight cycle of even length is a triangulation of the cylinder, while a tight cycle of odd length is a triangulation of the M\"obius strip,  see Figure \ref{fig:topcycle} for an illustration.  Another topological cycle of particular interest comes from double pyramids. Let $x,x',y_{1},\dots,y_{s}$ be the vertices of a double-pyramid, with edges $e_{i}=xy_{i}y_{i+1},f_{i}=x'y_{i}y_{i+1}$ for $i=1,\dots,s$, where $s\geq 4$. Then for any $3\leq r\leq s-1$, the sequence $e_{1},\dots,e_{r},f_{r},\dots,f_{s},f_{1}$ is a proper ordering of a topological cycle.

Say that a topological cycle $\mathcal{C}$ is \emph{torus-like}, if either 
\begin{itemize}
    \item $\mathcal{C}$ has even number of edges, and $\mathcal{C}$ is the triangulation of the cylinder, or
    \item $\mathcal{C}$ has odd number of edges, and $\mathcal{C}$ is the triangulation of the M\"obius strip.
\end{itemize}
If $\mathcal{C}$ is not torus-like, then say that $\mathcal{C}$ is \emph{Klein bottle-like}. Note that, for example, every tight cycle is torus-like.

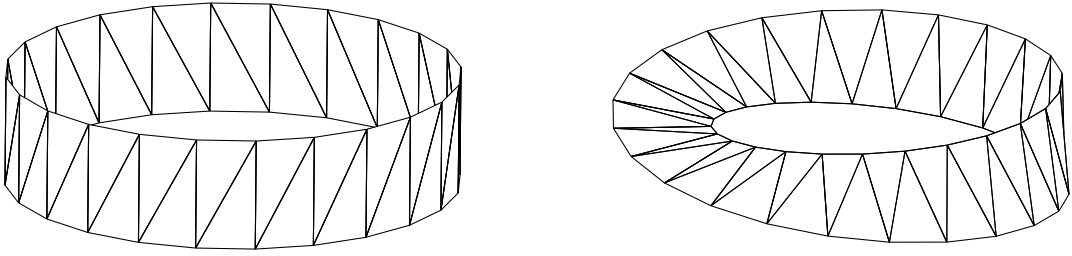
\begin{figure}
	\centering
\begin{tikzpicture}
\draw (-1.0125,-0.648825) -- (-1,0.783927) ;
\draw (-1,0.783927) -- (-1,0.783928) ;
\draw (-1,0.783928) -- (-1.17948,1.01842) ;
\draw (-1.12224,0.446908) -- (-1.17948,1.01842) ;
\draw (-1.17948,1.01842) -- (-1.18501,0.383991) ;
\draw (-1,0.783928) -- (-1,0.783927) ;
\draw (-1,0.783926) -- (-1.0125,-0.648825) ;
\draw (-1.12224,0.446908) -- (-1.17948,1.01842) ;
\draw (-1.17948,1.01842) -- (-1.12224,0.446908) ;
\draw (-1.18501,0.383991) -- (-1.17948,1.01842) ;
\draw (-1.17948,1.01842) -- (-1.55074,1.23234) ;
\draw (-1.34187,0.273656) -- (-1.55074,1.23234) ;
\draw (-1.55074,1.23234) -- (-1.56003,0.168092) ;
\draw (-1.17948,1.01842) -- (-1.18501,0.383991) ;
\draw (-1.34187,0.273656) -- (-1.55074,1.23234) ;
\draw (-1.55074,1.23234) -- (-1.34187,0.273656) ;
\draw (-1.56003,0.168093) -- (-1.55074,1.23234) ;
\draw (-1.55074,1.23234) -- (-2.0885,1.41109) ;
\draw (-1.66663,0.116799) -- (-2.0885,1.41109) ;
\draw (-2.0885,1.41109) -- (-2.10084,-0.00409573) ;
\draw (-1.55074,1.23234) -- (-1.56003,0.168092) ;
\draw (-1.66663,0.116799) -- (-2.0885,1.41109) ;
\draw (-2.0885,1.41109) -- (-1.66663,0.116799) ;
\draw (-2.10084,-0.00409484) -- (-2.0885,1.41109) ;
\draw (-2.0885,1.41109) -- (-2.75609,1.5425) ;
\draw (-2.10756,-0.00596791) -- (-2.75609,1.5425) ;
\draw (-2.75609,1.5425) -- (-2.76858,0.109743) ;
\draw (-2.0885,1.41109) -- (-2.10084,-0.00409573) ;
\draw (-2.10756,-0.00596791) -- (-2.75609,1.5425) ;
\draw (-2.75609,1.5425) -- (-2.10756,-0.00596791) ;
\draw (-2.13787,-0.0144052) -- (-2.76858,0.109743) ;
\draw (-2.76858,0.109743) -- (-2.75609,1.5425) ;
\draw (-2.75609,1.5425) -- (-3.50802,1.61761) ;
\draw (-2.76858,0.109743) -- (-3.50802,1.61761) ;
\draw (-3.50802,1.61761) -- (-3.52052,0.184852) ;
\draw (-2.75609,1.5425) -- (-2.76858,0.109743) ;
\draw (-2.76858,0.109743) -- (-3.50802,1.61761) ;
\draw (-3.50802,1.61761) -- (-2.76858,0.109743) ;
\draw (-2.76858,0.109743) -- (-3.52052,0.184852) ;
\draw (-3.52052,0.184852) -- (-3.50802,1.61761) ;
\draw (-3.50802,1.61761) -- (-4.29306,1.6313) ;
\draw (-3.52052,0.184852) -- (-4.29306,1.6313) ;
\draw (-4.29306,1.6313) -- (-4.30556,0.198544) ;
\draw (-3.50802,1.61761) -- (-3.52052,0.184852) ;
\draw (-3.52052,0.184852) -- (-4.29306,1.6313) ;
\draw (-4.29306,1.6313) -- (-3.52052,0.184852) ;
\draw (-3.52052,0.184852) -- (-4.30556,0.198544) ;
\draw (-4.30556,0.198544) -- (-4.29306,1.6313) ;
\draw (-4.29306,1.6313) -- (-5.0577,1.58264) ;
\draw (-4.30556,0.198544) -- (-5.0577,1.58264) ;
\draw (-5.0577,1.58264) -- (-5.0702,0.149885) ;
\draw (-4.29306,1.6313) -- (-4.30556,0.198544) ;
\draw (-4.30556,0.198544) -- (-5.0577,1.58264) ;
\draw (-5.0577,1.58264) -- (-4.30556,0.198544) ;
\draw (-4.30556,0.198544) -- (-5.0702,0.149885) ;
\draw (-5.0702,0.149885) -- (-5.0577,1.58264) ;
\draw (-5.0577,1.58264) -- (-5.74984,1.47494) ;
\draw (-5.0702,0.149885) -- (-5.74984,1.47494) ;
\draw (-5.74984,1.47494) -- (-5.76233,0.042192) ;
\draw (-5.0577,1.58264) -- (-5.0702,0.149885) ;
\draw (-5.0702,0.149885) -- (-5.74984,1.47494) ;
\draw (-5.74984,1.47494) -- (-5.0702,0.149885) ;
\draw (-5.0702,0.149885) -- (-5.76233,0.042192) ;
\draw (-5.76233,0.042192) -- (-5.74984,1.47494) ;
\draw (-5.74984,1.47494) -- (-6.3223,1.31556) ;
\draw (-5.76233,0.042192) -- (-6.3223,1.31556) ;
\draw (-6.3223,1.31556) -- (-6.33233,0.165701) ;
\draw (-5.74984,1.47494) -- (-5.76233,0.042192) ;
\draw (-5.76233,0.042192) -- (-6.3223,1.31556) ;
\draw (-6.3223,1.31556) -- (-5.76233,0.042192) ;
\draw (-5.76233,0.042192) -- (-5.86213,0.0144052) ;
\draw (-6.33233,0.165701) -- (-6.3223,1.31556) ;
\draw (-6.3223,1.31556) -- (-6.73607,1.11533) ;
\draw (-6.43853,0.201434) -- (-6.73607,1.11533) ;
\draw (-6.73607,1.11533) -- (-6.74252,0.376581) ;
\draw (-6.3223,1.31556) -- (-6.33233,0.165701) ;
\draw (-6.43853,0.201434) -- (-6.73607,1.11533) ;
\draw (-6.73607,1.11533) -- (-6.43853,0.201435) ;
\draw (-6.74252,0.376582) -- (-6.73607,1.11533) ;
\draw (-6.73607,1.11533) -- (-6.96296,0.887926) ;
\draw (-6.90013,0.534662) -- (-6.96296,0.887926) ;
\draw (-6.96296,0.887926) -- (-6.9653,0.619815) ;
\draw (-6.73607,1.11533) -- (-6.74252,0.376581) ;
\draw (-6.90013,0.534662) -- (-6.96296,0.887926) ;
\draw (-6.96296,0.887926) -- (-6.90013,0.534662) ;
\draw (-6.9653,0.619815) -- (-6.96296,0.887926) ;
\draw (-6.96296,0.887926) -- (-6.9875,0.648825) ;
\draw (-6.9875,0.648825) -- (-7,-0.783927) ;
\draw (-6.96296,0.887926) -- (-6.9653,0.619815) ;
\draw (-7,-0.783927) -- (-6.9875,0.648825) ;
\draw (-6.9875,0.648825) -- (-6.80803,0.414328) ;
\draw (-7,-0.783927) -- (-6.80803,0.414328) ;
\draw (-6.80803,0.414328) -- (-6.82052,-1.01842) ;
\draw (-6.9875,0.648825) -- (-7,-0.783927) ;
\draw (-7,-0.783927) -- (-6.80803,0.414328) ;
\draw (-6.80803,0.414328) -- (-7,-0.783927) ;
\draw (-7,-0.783927) -- (-6.82052,-1.01842) ;
\draw (-6.82052,-1.01842) -- (-6.80803,0.414328) ;
\draw (-6.80803,0.414328) -- (-6.43676,0.200416) ;
\draw (-6.82052,-1.01842) -- (-6.43676,0.200416) ;
\draw (-6.43676,0.200416) -- (-6.44926,-1.23234) ;
\draw (-6.80803,0.414328) -- (-6.82052,-1.01842) ;
\draw (-6.82052,-1.01842) -- (-6.43676,0.200416) ;
\draw (-6.43676,0.200416) -- (-6.82052,-1.01842) ;
\draw (-6.82052,-1.01842) -- (-6.44926,-1.23234) ;
\draw (-6.44926,-1.23234) -- (-6.43676,0.200416) ;
\draw (-6.43676,0.200416) -- (-5.89901,0.0216642) ;
\draw (-6.44926,-1.23234) -- (-5.89901,0.0216642) ;
\draw (-5.89901,0.0216642) -- (-5.91151,-1.41109) ;
\draw (-6.43676,0.200416) -- (-6.44926,-1.23234) ;
\draw (-6.44926,-1.23234) -- (-5.89901,0.0216642) ;
\draw (-5.89901,0.0216642) -- (-6.44926,-1.23234) ;
\draw (-6.44926,-1.23234) -- (-5.91151,-1.41109) ;
\draw (-5.91151,-1.41109) -- (-5.89901,0.0216642) ;
\draw (-5.89901,0.0216642) -- (-5.23142,-0.109743) ;
\draw (-5.91151,-1.41109) -- (-5.23142,-0.109743) ;
\draw (-5.23142,-0.109743) -- (-5.24392,-1.5425) ;
\draw (-5.89901,0.0216642) -- (-5.91151,-1.41109) ;
\draw (-5.91151,-1.41109) -- (-5.23142,-0.109743) ;
\draw (-5.23142,-0.109743) -- (-5.91151,-1.41109) ;
\draw (-5.91151,-1.41109) -- (-5.24392,-1.5425) ;
\draw (-5.24392,-1.5425) -- (-5.23142,-0.109743) ;
\draw (-5.23142,-0.109743) -- (-4.47948,-0.184852) ;
\draw (-5.24392,-1.5425) -- (-4.47948,-0.184852) ;
\draw (-4.47948,-0.184852) -- (-4.49198,-1.61761) ;
\draw (-5.23142,-0.109743) -- (-5.24392,-1.5425) ;
\draw (-5.24392,-1.5425) -- (-4.47948,-0.184852) ;
\draw (-4.47948,-0.184852) -- (-5.24392,-1.5425) ;
\draw (-5.24392,-1.5425) -- (-4.49198,-1.61761) ;
\draw (-4.49198,-1.61761) -- (-4.47948,-0.184852) ;
\draw (-4.47948,-0.184852) -- (-3.69444,-0.198544) ;
\draw (-4.49198,-1.61761) -- (-3.69444,-0.198544) ;
\draw (-3.69444,-0.198544) -- (-3.70694,-1.6313) ;
\draw (-4.47948,-0.184852) -- (-4.49198,-1.61761) ;
\draw (-4.49198,-1.61761) -- (-3.69444,-0.198544) ;
\draw (-3.69444,-0.198544) -- (-4.49198,-1.61761) ;
\draw (-4.49198,-1.61761) -- (-3.70694,-1.6313) ;
\draw (-3.70694,-1.6313) -- (-3.69444,-0.198544) ;
\draw (-3.69444,-0.198544) -- (-2.9298,-0.149885) ;
\draw (-3.70694,-1.6313) -- (-2.9298,-0.149885) ;
\draw (-2.9298,-0.149885) -- (-2.9423,-1.58264) ;
\draw (-3.69444,-0.198544) -- (-3.70694,-1.6313) ;
\draw (-3.70694,-1.6313) -- (-2.9298,-0.149885) ;
\draw (-2.9298,-0.149885) -- (-3.70694,-1.6313) ;
\draw (-3.70694,-1.6313) -- (-2.9423,-1.58264) ;
\draw (-2.9423,-1.58264) -- (-2.9298,-0.149885) ;
\draw (-2.9298,-0.149885) -- (-2.23767,-0.042192) ;
\draw (-2.9423,-1.58264) -- (-2.23767,-0.042192) ;
\draw (-2.23767,-0.042192) -- (-2.25016,-1.47494) ;
\draw (-2.9298,-0.149885) -- (-2.9423,-1.58264) ;
\draw (-2.9423,-1.58264) -- (-2.23767,-0.042192) ;
\draw (-2.23767,-0.042192) -- (-2.9423,-1.58264) ;
\draw (-2.9423,-1.58264) -- (-2.25016,-1.47494) ;
\draw (-2.25016,-1.47494) -- (-2.23767,-0.042192) ;
\draw (-2.23767,-0.042192) -- (-1.66521,0.117196) ;
\draw (-2.25016,-1.47494) -- (-1.66521,0.117196) ;
\draw (-1.66521,0.117196) -- (-1.6777,-1.31556) ;
\draw (-2.23767,-0.042192) -- (-2.25016,-1.47494) ;
\draw (-2.25016,-1.47494) -- (-1.66521,0.117196) ;
\draw (-1.66521,0.117196) -- (-2.25016,-1.47494) ;
\draw (-2.25016,-1.47494) -- (-1.6777,-1.31556) ;
\draw (-1.6777,-1.31556) -- (-1.66521,0.117196) ;
\draw (-1.66521,0.117196) -- (-1.25143,0.317416) ;
\draw (-1.6777,-1.31556) -- (-1.25143,0.317416) ;
\draw (-1.25143,0.317416) -- (-1.26393,-1.11534) ;
\draw (-1.66521,0.117196) -- (-1.6777,-1.31556) ;
\draw (-1.6777,-1.31556) -- (-1.25143,0.317416) ;
\draw (-1.25143,0.317416) -- (-1.6777,-1.31556) ;
\draw (-1.6777,-1.31556) -- (-1.26393,-1.11534) ;
\draw (-1.26393,-1.11534) -- (-1.25143,0.317416) ;
\draw (-1.25143,0.317416) -- (-1.02454,0.544826) ;
\draw (-1.26393,-1.11534) -- (-1.02454,0.544826) ;
\draw (-1.02454,0.544826) -- (-1.03704,-0.887926) ;
\draw (-1.25143,0.317416) -- (-1.26393,-1.11534) ;
\draw (-1.26393,-1.11534) -- (-1.02454,0.544826) ;
\draw (-1.02454,0.544826) -- (-1.26393,-1.11534) ;
\draw (-1.26393,-1.11534) -- (-1.03704,-0.887926) ;
\draw (-1.03704,-0.887926) -- (-1.02454,0.544826) ;
\draw (-1.02454,0.544826) -- (-1,0.783927) ;
\draw (-1.03704,-0.887926) -- (-1,0.783927) ;
\draw (-1,0.783927) -- (-1.0125,-0.648825) ;
\draw (-1.02454,0.544826) -- (-1.03704,-0.887926) ;
\draw (-1.03704,-0.887926) -- (-1,0.783927) ;
\draw (-1,0.783927) -- (-1.03704,-0.887926) ;
\draw (-1.03704,-0.887926) -- (-1.0125,-0.648825) ;
\draw (6.91189,0.69707) -- (6.75955,0.935535) ;
\draw (6.75955,0.935535) -- (6.74048,0.288655) ;
\draw (6.91189,0.69707) -- (6.86754,0.437425) ;
\draw (6.75955,0.935535) -- (6.42638,1.15611) ;
\draw (6.42638,1.15611) -- (6.36162,0.0352299) ;
\draw (6.75955,0.935535) -- (6.45679,0.0919694) ;
\draw (6.42638,1.15611) -- (5.92502,1.34117) ;
\draw (5.92502,1.34117) -- (5.86456,-0.0254425) ;
\draw (6.42638,1.15611) -- (5.86456,-0.0254425) ;
\draw (5.86456,-0.0254425) -- (6.02722,-0.0887141) ;
\draw (5.92502,1.34117) -- (5.28246,1.47393) ;
\draw (5.28246,1.47393) -- (5.31033,0.124293) ;
\draw (5.92502,1.34117) -- (5.31033,0.124293) ;
\draw (5.31033,0.124293) -- (5.86456,-0.0254425) ;
\draw (5.28246,1.47393) -- (4.53889,1.54007) ;
\draw (4.53889,1.54007) -- (4.72467,0.231438) ;
\draw (5.28246,1.47393) -- (4.72467,0.231438) ;
\draw (4.72467,0.231438) -- (5.31033,0.124293) ;
\draw (4.53889,1.54007) -- (3.74514,1.52917) ;
\draw (3.74514,1.52917) -- (4.1457,0.294923) ;
\draw (4.53889,1.54007) -- (4.1457,0.294923) ;
\draw (4.1457,0.294923) -- (4.72467,0.231438) ;
\draw (3.74514,1.52917) -- (2.95879,1.43599) ;
\draw (2.95879,1.43599) -- (3.60727,0.316678) ;
\draw (3.74514,1.52917) -- (3.60727,0.316678) ;
\draw (3.60727,0.316678) -- (4.1457,0.294923) ;
\draw (2.95879,1.43599) -- (2.23942,1.26128) ;
\draw (2.23942,1.26128) -- (3.13753,0.300984) ;
\draw (2.95879,1.43599) -- (3.13753,0.300984) ;
\draw (3.13753,0.300984) -- (3.60727,0.316678) ;
\draw (2.23942,1.26128) -- (1.64337,1.01204) ;
\draw (1.64337,1.01204) -- (2.75831,0.253788) ;
\draw (2.23942,1.26128) -- (2.75831,0.253788) ;
\draw (2.75831,0.253788) -- (3.13753,0.300984) ;
\draw (1.64337,1.01204) -- (1.21857,0.701229) ;
\draw (1.21857,0.701229) -- (2.48512,0.182065) ;
\draw (1.64337,1.01204) -- (2.48512,0.182065) ;
\draw (2.48512,0.182065) -- (2.75831,0.253788) ;
\draw (1.21857,0.701229) -- (1,0.346864) ;
\draw (1,0.346864) -- (2.32761,0.0933176) ;
\draw (1.21857,0.701229) -- (2.32761,0.0933176) ;
\draw (2.32761,0.0933176) -- (2.48512,0.182065) ;
\draw (1,0.346864) -- (1.0063,-0.0293872) ;
\draw (1.0063,-0.0293872) -- (2.29018,-0.00477523) ;
\draw (1,0.346864) -- (2.29018,-0.00477523) ;
\draw (2.29018,-0.00477523) -- (2.32761,0.0933176) ;
\draw (1.0063,-0.0293872) -- (1.23787,-0.403993) ;
\draw (1.23787,-0.403993) -- (2.37267,-0.104553) ;
\draw (1.0063,-0.0293872) -- (2.37267,-0.104553) ;
\draw (2.37267,-0.104553) -- (2.29018,-0.00477523) ;
\draw (1.23787,-0.403993) -- (1.67678,-0.753537) ;
\draw (1.67678,-0.753537) -- (2.57067,-0.198466) ;
\draw (1.23787,-0.403993) -- (2.57067,-0.198466) ;
\draw (2.57067,-0.198466) -- (2.37267,-0.104553) ;
\draw (1.67678,-0.753537) -- (2.28832,-1.05669) ;
\draw (2.28832,-1.05669) -- (2.87574,-0.279108) ;
\draw (1.67678,-0.753537) -- (2.87574,-0.279108) ;
\draw (2.87574,-0.279108) -- (2.57067,-0.198466) ;
\draw (2.28832,-1.05669) -- (3.02429,-1.29599) ;
\draw (3.02429,-1.29599) -- (3.27514,-0.339299) ;
\draw (2.28832,-1.05669) -- (3.27514,-0.339299) ;
\draw (3.27514,-0.339299) -- (2.87574,-0.279108) ;
\draw (3.02429,-1.29599) -- (3.82741,-1.45917) ;
\draw (3.82741,-1.45917) -- (3.75161,-0.372242) ;
\draw (3.02429,-1.29599) -- (3.75161,-0.372242) ;
\draw (3.75161,-0.372242) -- (3.27514,-0.339299) ;
\draw (3.82741,-1.45917) -- (4.63658,-1.54007) ;
\draw (4.63658,-1.54007) -- (4.28302,-0.371824) ;
\draw (3.82741,-1.45917) -- (4.28302,-0.371824) ;
\draw (4.28302,-0.371824) -- (3.75161,-0.372242) ;
\draw (4.63658,-1.54007) -- (5.39225,-1.53883) ;
\draw (5.39225,-1.53883) -- (4.84238,-0.333068) ;
\draw (4.63658,-1.54007) -- (4.84238,-0.333068) ;
\draw (4.84238,-0.333068) -- (4.28302,-0.371824) ;
\draw (5.39225,-1.53883) -- (6.04151,-1.46161) ;
\draw (6.04151,-1.46161) -- (5.39826,-0.252701) ;
\draw (5.39225,-1.53883) -- (5.39826,-0.252701) ;
\draw (5.39826,-0.252701) -- (4.84238,-0.333068) ;
\draw (6.04151,-1.46161) -- (6.54226,-1.31966) ;
\draw (6.54226,-1.31966) -- (5.91592,-0.129778) ;
\draw (6.04151,-1.46161) -- (5.91592,-0.129778) ;
\draw (5.91592,-0.129778) -- (5.39826,-0.252701) ;
\draw (6.54226,-1.31966) -- (6.86615,-1.12807) ;
\draw (6.86615,-1.12807) -- (6.3591,0.0337315) ;
\draw (6.54226,-1.31966) -- (6.3591,0.0337315) ;
\draw (6.3591,0.0337315) -- (5.91592,-0.129778) ;
\draw (6.86615,-1.12807) -- (7,-0.904183) ;
\draw (7,-0.904183) -- (6.69253,0.23251) ;
\draw (6.86615,-1.12807) -- (6.69253,0.23251) ;
\draw (6.69253,0.23251) -- (6.3591,0.0337315) ;
\draw (7,-0.904183) -- (6.88486,0.457702) ;
\draw (6.88486,0.457702) -- (6.69253,0.23251) ;
\draw (6.93161,-0.0952095) -- (6.91189,0.69707) ;
\draw (6.91189,0.69707) -- (6.88486,0.457702) ;
\draw (6.91189,0.69707) -- (6.86754,0.437425) ;
\draw (6.74048,0.288655) -- (6.75955,0.935535) ;
\draw (6.91189,0.69707) -- (6.93161,-0.0952095) ;
\draw (6.75955,0.935535) -- (6.45679,0.0919694) ;
\draw (6.36162,0.0352308) -- (6.42638,1.15611) ;
\draw (6.75955,0.935535) -- (6.74048,0.288655) ;
\draw (6.42638,1.15611) -- (5.86456,-0.0254425) ;
\draw (5.86456,-0.0254425) -- (5.92502,1.34117) ;
\draw (6.42638,1.15611) -- (6.36162,0.0352299) ;
\draw (6.02722,-0.0887141) -- (5.86456,-0.0254425) ;
\draw (5.92502,1.34117) -- (5.31033,0.124293) ;
\draw (5.31033,0.124293) -- (5.28246,1.47393) ;
\draw (5.92502,1.34117) -- (5.86456,-0.0254425) ;
\draw (5.86456,-0.0254425) -- (5.31033,0.124293) ;
\draw (5.28246,1.47393) -- (4.72467,0.231438) ;
\draw (4.72467,0.231438) -- (4.53889,1.54007) ;
\draw (5.28246,1.47393) -- (5.31033,0.124293) ;
\draw (5.31033,0.124293) -- (4.72467,0.231438) ;
\draw (4.53889,1.54007) -- (4.1457,0.294923) ;
\draw (4.1457,0.294923) -- (3.74514,1.52917) ;
\draw (4.53889,1.54007) -- (4.72467,0.231438) ;
\draw (4.72467,0.231438) -- (4.1457,0.294923) ;
\draw (3.74514,1.52917) -- (3.60727,0.316678) ;
\draw (3.60727,0.316678) -- (2.95879,1.43599) ;
\draw (3.74514,1.52917) -- (4.1457,0.294923) ;
\draw (4.1457,0.294923) -- (3.60727,0.316678) ;
\draw (2.95879,1.43599) -- (3.13753,0.300984) ;
\draw (3.13753,0.300984) -- (2.23942,1.26128) ;
\draw (2.95879,1.43599) -- (3.60727,0.316678) ;
\draw (3.60727,0.316678) -- (3.13753,0.300984) ;
\draw (2.23942,1.26128) -- (2.75831,0.253788) ;
\draw (2.75831,0.253788) -- (1.64337,1.01204) ;
\draw (2.23942,1.26128) -- (3.13753,0.300984) ;
\draw (3.13753,0.300984) -- (2.75831,0.253788) ;
\draw (1.64337,1.01204) -- (2.48512,0.182065) ;
\draw (2.48512,0.182065) -- (1.21857,0.701229) ;
\draw (1.64337,1.01204) -- (2.75831,0.253788) ;
\draw (2.75831,0.253788) -- (2.48512,0.182065) ;
\draw (1.21857,0.701229) -- (2.32761,0.0933176) ;
\draw (2.32761,0.0933176) -- (1,0.346864) ;
\draw (1.21857,0.701229) -- (2.48512,0.182065) ;
\draw (2.48512,0.182065) -- (2.32761,0.0933176) ;
\draw (1,0.346864) -- (2.29018,-0.00477523) ;
\draw (2.29018,-0.00477523) -- (1.0063,-0.0293872) ;
\draw (1,0.346864) -- (2.32761,0.0933176) ;
\draw (2.32761,0.0933176) -- (2.29018,-0.00477523) ;
\draw (1.0063,-0.0293872) -- (2.37267,-0.104553) ;
\draw (2.37267,-0.104553) -- (1.23787,-0.403993) ;
\draw (1.0063,-0.0293872) -- (2.29018,-0.00477523) ;
\draw (2.29018,-0.00477523) -- (2.37267,-0.104553) ;
\draw (1.23787,-0.403993) -- (2.57067,-0.198466) ;
\draw (2.57067,-0.198466) -- (1.67678,-0.753537) ;
\draw (1.23787,-0.403993) -- (2.37267,-0.104553) ;
\draw (2.37267,-0.104553) -- (2.57067,-0.198466) ;
\draw (1.67678,-0.753537) -- (2.87574,-0.279108) ;
\draw (2.87574,-0.279108) -- (2.28832,-1.05669) ;
\draw (1.67678,-0.753537) -- (2.57067,-0.198466) ;
\draw (2.57067,-0.198466) -- (2.87574,-0.279108) ;
\draw (2.28832,-1.05669) -- (3.27514,-0.339299) ;
\draw (3.27514,-0.339299) -- (3.02429,-1.29599) ;
\draw (2.28832,-1.05669) -- (2.87574,-0.279108) ;
\draw (2.87574,-0.279108) -- (3.27514,-0.339299) ;
\draw (3.02429,-1.29599) -- (3.75161,-0.372242) ;
\draw (3.75161,-0.372242) -- (3.82741,-1.45917) ;
\draw (3.02429,-1.29599) -- (3.27514,-0.339299) ;
\draw (3.27514,-0.339299) -- (3.75161,-0.372242) ;
\draw (3.82741,-1.45917) -- (4.28302,-0.371824) ;
\draw (4.28302,-0.371824) -- (4.63658,-1.54007) ;
\draw (3.82741,-1.45917) -- (3.75161,-0.372242) ;
\draw (3.75161,-0.372242) -- (4.28302,-0.371824) ;
\draw (4.63658,-1.54007) -- (4.84238,-0.333068) ;
\draw (4.84238,-0.333068) -- (5.39225,-1.53883) ;
\draw (4.63658,-1.54007) -- (4.28302,-0.371824) ;
\draw (4.28302,-0.371824) -- (4.84238,-0.333068) ;
\draw (5.39225,-1.53883) -- (5.39826,-0.252701) ;
\draw (5.39826,-0.252701) -- (6.04151,-1.46161) ;
\draw (5.39225,-1.53883) -- (4.84238,-0.333068) ;
\draw (4.84238,-0.333068) -- (5.39826,-0.252701) ;
\draw (6.04151,-1.46161) -- (5.91592,-0.129778) ;
\draw (5.91592,-0.129778) -- (6.54226,-1.31966) ;
\draw (6.04151,-1.46161) -- (5.39826,-0.252701) ;
\draw (5.39826,-0.252701) -- (5.91592,-0.129778) ;
\draw (6.54226,-1.31966) -- (6.3591,0.0337315) ;
\draw (6.3591,0.0337315) -- (6.86615,-1.12807) ;
\draw (6.54226,-1.31966) -- (5.91592,-0.129778) ;
\draw (5.91592,-0.129778) -- (6.3591,0.0337315) ;
\draw (6.86615,-1.12807) -- (6.69253,0.23251) ;
\draw (6.69253,0.23251) -- (7,-0.904183) ;
\draw (6.86615,-1.12807) -- (6.3591,0.0337315) ;
\draw (6.3591,0.0337315) -- (6.69253,0.23251) ;
\draw (7,-0.904183) -- (6.88486,0.457702) ;
\draw (7,-0.904183) -- (6.69253,0.23251) ;
\draw (6.69253,0.23251) -- (6.88486,0.457702) ;
\draw (6.88486,0.457702) -- (6.91189,0.69707) ;
\end{tikzpicture}
\caption{A tight cycle of length 48 (left) and a tight cycle of length 49 (right).}
\label{fig:topcycle}
\end{figure}

The next lemma tells us that if we glue spheres along a topological cycle, then we either get a torus or a Klein bottle (which is the sphere with two cross-caps).

\begin{lemma}\label{lemma:toplogy}
Let $\mathcal{C}$ be a topological cycle, and let $e_{1},\dots,e_{r}$ be a proper ordering of the edges. Let 3-uniform hypergraphs $\mathcal{S}_{1},\dots,\mathcal{S}_{r}$ be triangulations of the sphere such that $e_i,e_{i+1}\in E(\mathcal{S}_{i})$ for $i=1,\dots,r$, where indices are meant modulo $r$, and $\mathcal{S}_{i}$ has no vertex common with $\mathcal{C}$ and $\mathcal{S}_{j}$ for $j\neq i$, with the possible exception of the vertices in $e_i\cup e_{i+1}$. Let $\mathcal{T}$ be the hypergraph with edge set $\bigcup_{i=1}^{r} (E(\mathcal{S}_{i})\setminus \{e_{i},e_{i+1}\})$. If $\mathcal{C}$ is torus-like, then $\mathcal{T}$ is homeomorphic to the torus, and if $\mathcal{C}$ is Klein bottle-like, then $\mathcal{T}$ is homeomorphic to the Klein bottle.
\end{lemma}

\begin{proof}
For $i=1,\dots,r$, let $\mathcal{D}_{i}$ be the triangulation of the disc we get after removing $e_{i}$ and $e_{i+1}$ from $\mathcal{S}_{i}$. For $i=1,\dots, r$, denote by $s_i\subset e_i$ the 2-edge lying on the boundary of $\mathcal C$. Remark that the boundary of $\mathcal D_i$ contains $s_i$ and $s_{i+1}$. Put $\varepsilon_i=1$, if these two 2-edges share a common vertex, otherwise put  $\varepsilon_i=-1$. Informally, $\varepsilon_i=-1$ if the 2-edges $s_i$ and $s_{i+1}$ of the boundary of $\mathcal D_i$ belong to 'the locally opposite sides of $\mathcal C$' and $\varepsilon_i=-1$ if they belong to 'the same side of $\mathcal C$'. It is easy to see that if $\varepsilon_1\dots \varepsilon_r=1$, then $\mathcal C$ is a triangulation of the cylinder $S^1\times [0,1]$ (because in these case the boundary of $\mathcal C$ consists of two disconnected parts), otherwise $\mathcal C$ is a triangulation of the M\"obius strip (because $\mathcal C$ has only one connected parts). Therefore, $\varepsilon_1\dots \varepsilon_r=(-1)^r$ if and only if $\mathcal C$ is torus-like.

First, let us check that $\mathcal{T}$ is a simplicial manifold, that is, for every $v\in V(\mathcal T)$ the link graph of $v$ in $\mathcal T$ is a cycle (and thus $\mathcal{T}$ is homeomorphic to a closed surface). There are two possible cases:

\textit{Case 1.} Suppose that $v\not\in V(\mathcal C)$, and thus $v\in V(\mathcal S_i)\setminus (e_i\cup e_{i+1})$ for some $i\in [r]$. Since $\mathcal S_i$ is a triangulation of the sphere and $v$ does not belong to other triangulations $\mathcal S_j$, the link graph of $v$ in $\mathcal{T}$ is a cycle.

\textit{Case 2}. Suppose that $v\in V(\mathcal{C})$, and let $e_a=w_aw_{a+1}v, e_{a+1}=w_{a+1}w_{a+2}v, \dots, e_b=w_{b}w_{b+1}v$ be all the edges of $\mathcal{C}$ containing $v$. Let us describe the link graph $T_v$ of $v$ in $\mathcal{T}$. 

Since $v$ belongs to $V(\mathcal S_{a-1}), \dots, V(\mathcal S_{b})$ and does not belong to any other $V(\mathcal S_i)$, the link graph of $v$ in $\mathcal T$ naturally falls into several parts:
\begin{enumerate}
    \item The link graph of $v$ in $\mathcal D_{a-1}=\mathcal S_{a-1}\setminus \{e_{a-1},e_a\}$ is a simple path connecting $w_a$ and $w_{a+1}$ passing through vertices of $V(\mathcal S_{a-1})\setminus (e_{a-1}\cup e_a)$.
    
    \item The link graph of $v$ in $\mathcal S_i\setminus\{e_i, e_{i+1}\}$ for $i=a,\dots, b-1$ is a simple path connecting $w_i$ and $w_{i+2}$ and passing through vertices of $V(\mathcal S_{i})\setminus (e_i\cup e_{i+1})$.
    
    \item The link graph of $v$ in $\mathcal S_{b}\setminus\{e_{b},e_{b+1}\}$ is a simple path connecting $w_b$ and $w_{b+1}$ passing through vertices of $V(\mathcal S_{b})\setminus (e_b\cup e_{b+1})$.
\end{enumerate}
From this observations, we conclude that the link graph of $v$ in $\mathcal T$ is a simple cycle; see Figure~\ref{fig:linkgraph}.

\begin{figure}
	\centering
\begin{tikzpicture}
\draw [line width=0.1pt,dashed] (-5.07,-3.37)-- (-3.69,-1.09);
\draw [line width=0.1pt,dashed] (-3.69,-1.09)-- (-1.47,-0.81);
\draw [line width=0.1pt,dashed] (-1.47,-0.81)-- (0.05,-1.25);
\draw [line width=0.1pt,dashed] (0.05,-1.25)-- (1.33,-2.81);
\draw [line width=0.1pt,dashed] (1.33,-2.81)-- (-1.57,-3.41);
\draw [line width=0.1pt,dashed] (-1.57,-3.41)-- (-5.07,-3.37);
\draw [line width=0.1pt,dashed] (-1.57,-3.41)-- (-3.69,-1.09);
\draw [line width=0.1pt,dashed] (-1.57,-3.41)-- (-1.47,-0.81);
\draw [line width=0.1pt,dashed] (-1.57,-3.41)-- (0.05,-1.25);
\node at (-1.57,-3.41) [below] {$v$};
\node at (-5.07,-3.37) [below] {$w_a$};
\node at (-3.69,-1.09) [above left] {$w_{a+1}$};
\node at  (0.05,-1.25) [above right] {$w_{b}$};
\node at  (1.33,-2.81) [below] {$w_{b+1}$};
\draw [shift={(-3.4327762123683816,-1.8610959513569632)},line width=0.3pt]  plot[domain=0.49165425657508394:3.8862268488850833,variable=\t]({1*2.226497958408714*cos(\t r)+0*2.226497958408714*sin(\t r)},{0*2.226497958408714*cos(\t r)+1*2.226497958408714*sin(\t r)});
\draw [shift={(-1.844259052480459,-1.7370553517307388)},line width=0.3pt]  plot[domain=0.25167024749169975:2.8044130495295976,variable=\t]({1*1.955873276455742*cos(\t r)+0*1.955873276455742*sin(\t r)},{0*1.955873276455742*cos(\t r)+1*1.955873276455742*sin(\t r)});
\draw [shift={(0.021866330390920598,-1.6813871374527114)},line width=0.3pt]  plot[domain=-0.7118586842213492:2.6129523658454996,variable=\t]({1*1.7277096663132092*cos(\t r)+0*1.7277096663132092*sin(\t r)},{0*1.7277096663132092*cos(\t r)+1*1.7277096663132092*sin(\t r)});
\draw [shift={(-3.7363922518159796,-2.6195520581113807)},line width=0.3pt]  plot[domain=1.5404750093257797:3.6541493279618464,variable=\t]({1*1.5302554491003513*cos(\t r)+0*1.5302554491003513*sin(\t r)},{0*1.5302554491003513*cos(\t r)+1*1.5302554491003513*sin(\t r)});
\draw [shift={(-0.049405204460966155,-2.636691449814126)},line width=0.3pt]  plot[domain=-0.12498515598369675:1.4992335762665745,variable=\t]({1*1.3902498234711378*cos(\t r)+0*1.3902498234711378*sin(\t r)},{0*1.3902498234711378*cos(\t r)+1*1.3902498234711378*sin(\t r)});
\end{tikzpicture}
\caption{The link graph of $v\in V(\mathcal C)$.}
\label{fig:linkgraph}
\end{figure}
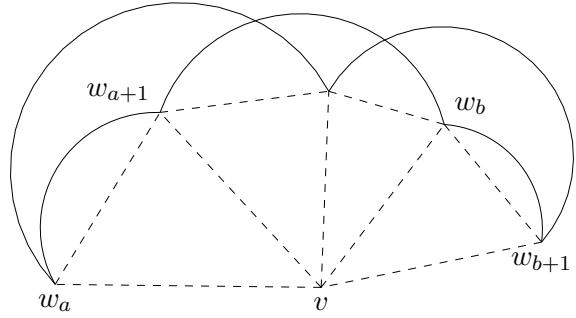

Second, let us study the orientability of $\mathcal T$. Recall that a simplicial $2$-manifold (or a triangulation)  is {\it orientable} iff we can orient the boundary of each triangle cyclically so that for any two triangles sharing an edge the orientations induced on the common edge are opposite. Remark that an orientation of the triangulation of $\mathcal D_i$ induces a cyclic orientation on the boundary of $D_i$, and in particular on $s_i$ and $s_{i+1}$. Inversely, an orientation of $s_i$ induces an orientation of the triangulation $\mathcal D_i$ of the disc. For a given orientation of $\mathcal D_i$, put $\delta_i=1$ if the orientation of $s_i=\{x,y\}$ induced by the orientation of $\mathcal D_i$ is $[x,y]$, where $x\in e_{i-1}$ and $y\in e_{i+1}$ (in this case, we say that  $s_i$ is {\it oriented forwards}), otherwise put $\delta_i=-1$ (and say that $s_i$ is {\it oriented backwards}).
For a given orientation of $\mathcal D_i\cup \mathcal D_{i+1}$, 
 we have $\delta_{i+1}=\delta_i(-\varepsilon_i)$. Indeed, if $s_i$ and $s_{i+1}$ share a vertex, then within $\mathcal D_i$ either both are oriented forwards, or both are oriented backwards. But the orientation of $s_{i+1}$ within $\mathcal D_{i+1}$ is opposite, and thus is opposite to the orientation of $s_i$ within $\mathcal D_i$. If $s_i$ and $s_{i+1}$ do not share a vertex (and thus locally `lie on the opposite sides of $\mathcal S$'), then within $\mathcal D_i$ exactly one of them is oriented forwards, and we conclude in the same way as before.

Clearly, $\mathcal T$ is orientable if and only if for each $i=1,\ldots, r$ there exists an orientation $\eta_i$ of $\mathcal D_i\cup \mathcal D_{i+1}$ such that $\eta_i$ and $\eta_{i+1}$ induce the same orientation of $D_{i+1}$. From the above, we conclude that such a set $\eta_1,\ldots, \eta_r$ of orientations exist iff
\[
\delta_1=\delta_1(-\varepsilon_1)(-\varepsilon_2)\dots (-\varepsilon_r),
\]
that is, 
\[
\varepsilon_1\dots \varepsilon_r=(-1)^r,
\]
which corresponds to the case of the torus-like topological cycle $\mathcal C$.



Lastly, we determine the Euler characteristic of $\mathcal{T}$. Clearly, $\mathcal{C}$ has Euler characteristic 0. Each disc $\mathcal{D}_{i}$ glued to $\mathcal{C}$ increases the Euler characteristic by 1, so after removing the $r$ edges $e_{1},\dots,e_{r}$, we end up with the Euler characteristic equal to $0$. Therefore, the Euler characteristic of $\mathcal{T}$ is $0$, which implies that $\mathcal{T}$ is homeomorphic to either the torus or the Klein bottle by the Classification theorem of closed surfaces. See Figure \ref{fig:torus} for an illustration of $\mathcal{T}$. 
\end{proof}

Somewhat surprisingly, it turns out that if $\mathcal{C}$ is a topological cycle that is also 3-partite, then $\mathcal{C}$ is always torus-like.

\begin{claim}\label{claim:3partite}
If $\mathcal{C}$ is a topological cycle and $\mathcal{C}$ is 3-partite, then $\mathcal{C}$ is torus-like.
\end{claim}

\begin{proof}
Let us use notation introduced in the proof of Lemma~\ref{lemma:toplogy}. 

Suppose that $V(\mathcal C)$ is partitioned into three sets $V_1$, $V_2$, and $V_3$. For $i=1,\dots,r $, put $\sigma_i=1$ if $s_i=\{x,y\}$, where  $x\in e_{i-1}\cap V_j$ and $y\in e_{i+1}\cap V_{j+1}$ for some $j\in\{1,2,3\}$, where index $j+1$ is meant modulo $3$. Otherwise, put $\sigma_{i}=-1$. It is easy to check that $\sigma_{i+1}=\sigma_i(-\varepsilon_i)$. As in the proof above, we have $\sigma_1 = \sigma_1(-\varepsilon_1)(-\varepsilon_2)\dots (-\varepsilon_r)$, which implies $\varepsilon_1\dots\varepsilon_r=(-1)^r$. That is, $\mathcal C$ is a torus-like topological cycle.
\end{proof}

\begin{figure}
	\centering
\includegraphics[scale=0.5]{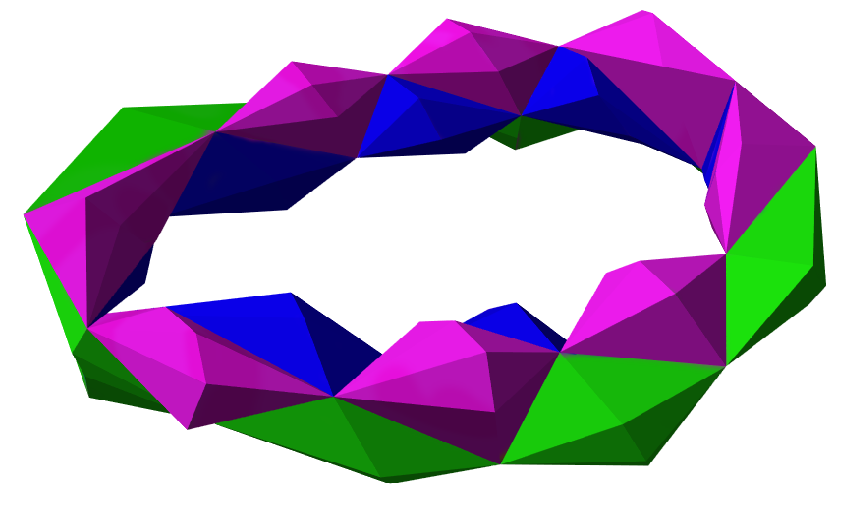}
\caption{An illustration of the torus we get after gluing spheres (more precisely double pyramids) to the neighboring edges of a tight cycle of length 24. We use the three colors to separate the spheres visually.}
\label{fig:torus}
\end{figure}

Now let us show how finding a topological cycle helps us building the triangulation of the torus.

\begin{lemma}\label{lemma:embedding_with_cyle}
Let $\mathcal{H}$ be a 3-uniform hypergraph, and let $\mathcal{C}$ be a topological cycle with proper edge ordering $e_1,\dots,e_{r}$. Suppose that $(e_i,e_{i+1})$ is $(\frac{1}{2r},\frac{1}{2r+1},2r,2r)$-semi-admissible in $\mathcal{H}$ for $i=1,\dots,r$. If $\mathcal{C}$ is torus-like, then $\mathcal{H}$ contains a torus, otherwise $\mathcal{H}$ contains a Klein bottle.
\end{lemma}

\begin{proof}
 By the definition of a $(\frac{1}{2r},\frac{1}{2r+1},2r,2r)$-semi-admissible pair, for $i=1,\dots,r$, there exist at least $2r$ edges $f_{i}^1,\dots,f_{i}^{2r}$ such that $f_{i}^{j}\cap e_{i}=f_{i}^{j}\cap e_{i+1}=e_{i}\cap e_{i+1}$ for $j=1,\dots,2r$, and $(e_{i},f_i^j)$ and $(e_{i+1},f_i^j)$ are both $(\frac{1}{2r},\frac{1}{2r+1},2r)$-admissible. Hence, for $i=1,\dots,r$, we can choose an edge $f_{i}$ among $f_i^1,\dots,f_i^{2r}$ such that  $|f_{i}\cap V(\mathcal{C})|=2$ and, moreover, the vertices  $f_{1}\setminus V(\mathcal{C}),\dots,f_{r}\setminus V(\mathcal{C})$ are pairwise distinct.
 
 Let $X=V(\mathcal{C})\cup\bigcup_{i=1}^{r} f_{i}$ and note that $|X|=2r$. Color each vertex in $V(\mathcal{H})\setminus X$ randomly and independently with one of the $2r$ colors $c_1,\dots,c_r,c_1',\dots,c_{r}'$. For $i=1,\dots,r$, let $A_{i}$ be the event that there exists a sphere $\mathcal{S}_{i}$ containing $e_{i}$ and $e_{i+1}$ such that all vertices of $\mathcal{S}_{i}$, with the exception of $e_{i}\cup f_{i}\cup e_{i+1}$, are colored with $c_{i}$ or $c_{i}'$. In the next paragraph, we show that $\mathbb{P}(A_{i})\geq 1-\frac{2}{2r+1}$. Assuming this inequality, with positive probability there exists a coloring for which the events $A_{1},\dots,A_{r}$ hold simultaneously. Note that the spheres $\mathcal{S}_{i}$ are pairwise vertex disjoint outside of $\mathcal{C}$, so we can apply Lemma \ref{lemma:toplogy} to conclude that $\mathcal{H}$ contains a torus or a Klein bottle.
 
 Now let us show that $\mathbb{P}(A_{i})\geq 1-\frac{2}{2r+1}$. Fix some $i\in \{1,\dots,r\}$, let $e_{i}=xyz, e_{i+1}=x'yz$ and $f=x''yz$. Let $B_{i}^1$ be the event that there exists a double pyramid $\mathcal{S}_{i}^{1}$  such that $\mathcal{S}_{i}^{1}$ contains $e_{i}$ and $f_{i}$, and every vertex in $V(\mathcal{S}_{i}^1)\setminus (e_{i}\cup f_{i})$ is colored $c_{i}$. Similarly, let $B_{i}^2$ be the event that there exists a double pyramid $\mathcal{S}_{i}^{2}$ such that $\mathcal{S}_{i}^{2}$ contains $e_{i+1}$ and $f_{i}$, and every vertex in $V(\mathcal{S}_{i}^2)\setminus (e_{i+1}\cup f_{i})$ is colored $c_{i}'$. We show that $\mathbb{P}(B_{i}^1)\geq 1-\frac{1}{2r+1}$. 
 Let $U$ be the set of vertices colored $c_{i}$. As the edge $yz$ is $(\frac{1}{2r},\frac{1}{2r+1},2r)$-admissible in the graph $G=\mathcal{H}_{x}\cap\mathcal{H}_{x''}$, the probability that there are $2r$ internally vertex disjoint paths between $y$ and $z$ (other than the edge $yz$) 
 in $G[U\cup X]$ is at least $1-\frac{1}{2r+1}$. But then at least one of these paths, say $P$, does not contain a vertex of $X$. The union of the path $P$ with the edge $\{y,z\}$ forms a cycle in $\mathcal{H}_{x}\cap \mathcal{H}_{x'}$, which gives a double pyramid $\mathcal{S}_{i}^{1}$ in $\mathcal{H}$ with the desired properties.
 The proof of $\mathbb{P}(B_{i}^2)\geq 1-\frac{1}{2r+1}$ is analogous. But then $\mathbb{P}(B_{i}^1\cap B_{i}^2)\geq 1-\frac{2}{2r+1}$. But if both $B_{i}^1$ and $B_{i}^2$ occur then, taking $\mathcal{S}_{i}$ to be the union of $\mathcal{S}_{i}^1$ and $\mathcal{S}_{i}^2$ with $f_{i}$ removed, $\mathcal{S}_{i}$ is a sphere satisfying the desired properties. This concludes the proof.
\end{proof}

\subsection{The extremal number of the torus}

After these preparations, we are almost done with the proof of Theorem \ref{thm:torus}. The last piece we are missing is that we can find short torus-like topological cycles in hypergraphs with $n$ vertices and $\Omega(n^{5/2})$ edges. Recently, it was proved by Sudakov and Tomon \cite{ST20} that if $\mathcal{H}$ has $n$ vertices and $n^{2+o(1)}$ edges, then $\mathcal{H}$ contains a tight cycle of length $O((\log n)^2)$. Also, by a much simpler argument, we can find a double pyramid of size $O(\log n)$ if $\mathcal{H}$ has $\Omega(n^{5/2})$ edges, which in turn contains a torus-like topological cycle of length $O(\log n)$. Using a topological cycle of such length, one can deduce that Theorem \ref{thm:torus} holds with the bound $O(n^{5/2}(\log n)^{3})$ instead of $O(n^{5/2})$. However, in order to prove the required bound $O(n^{5/2})$, we need to find a topological cycle of constant length. Conlon \cite{C10} proposed the conjecture that there exists a constant $c>0$ such that the extremal number of the tight cycle of length $3k$ is at most $O(n^{2+c/k})$. An unpublished result of Verstra\"ete \cite{V} states that the extremal number of the tight cycle of length 24 is $O(n^{5/2})$, which would be perfect for our purposes. However, in order to not rely on an unpublished result, we prove the following more general theorem, which can be viewed as a relaxation of Conlon's conjecture.

\begin{theorem}\label{thm:cycles}
Let $k\geq 2$ be an integer. If $\mathcal{H}$ is a 3-uniform hypergraph on $n$ vertices which contains no torus-like topological cycle of length at most $6k$, then $\mathcal{H}$ has at most $n^{2+2/(k-1)+o(1)}$ edges.
\end{theorem}

As the proof of this theorem is a bit of a detour, we present it later in Section \ref{sect:cycles}. One might wonder whether it is also possible to find a short Klein bottle-like topological cycle in $\mathcal{H}$ if $E(\mathcal{H})=\Omega(|V(\mathcal{H})|^{5/2})$, as this would allow us to find a triangulation of the Klein bottle. Unfortunately, this is not the case. Indeed, the complete 3-partite 3-uniform hypergraph with vertex classes of size $n/3$ has $\Omega(n^{3})$ edges and contains no Klein bottle-like topological cycle by Claim \ref{claim:3partite}.
 
\begin{proof}[Proof of Theorem \ref{thm:torus}]
Let $r=36$, then by Theorem \ref{thm:cycles}, there exists a constant $c_1$ such that every 3-uniform hypergraph on $n$ vertices with at least $c_1 n^{5/2}$ edges contains a torus like topological cycle of length at most $r$. We show that $c=\max\{2c_1,24\cdot r^2(2r+1)\}$ suffices.
 
 Let $\mathcal{H}$ be a 3-uniform hypergraph with at least $cn^{5/2}$ edges. Then by Lemma \ref{lemma:3admissible}, $\mathcal{H}$ contains a set $F$ of at least $c_1 n^{5/2}$ edges such that any neighboring pair of edges in $F$ is $(\frac{1}{2r},\frac{1}{2r+1},2r,2r)$-semi-admissible in $\mathcal{H}$. Also,  $F$ contains a torus-like topological cycle of length at most $r$. But then Lemma \ref{lemma:embedding_with_cyle} tells us that $\mathcal{H}$ contains a triangulation of the torus.
\end{proof} 

\section{Orientable surfaces --- Upper bound}\label{sect:surface}

In this section, we prove Theorem \ref{thm:orientable}. If $\mathcal{S}$ is a closed orientable surface of genus $g$, then $\mathcal{S}$ is homeomorphic to $g$ copies of the torus glued to each other one-by-one. We show the more general result that if we glue two hypergraphs $\mathcal{T}$ and $\mathcal{T}'$ together, then the extremal number of the resulting hypergraph is $O(\mbox{ex}(n,\mathcal{T})+\mbox{ex}(n,\mathcal{T}'))$. This might be of independent interest.

More precisely, let $\mathcal{T}$ and $\mathcal{T}'$ be two $r$-uniform hypergraphs. Define $\mathcal{T}\oplus\mathcal{T}'$ to be the family of hypergraphs we get after gluing some edge $e\in\mathcal{T}$ to some edge $e'\in\mathcal{T}'$. Also, given two families $\mathcal{F}$ and $\mathcal{F}'$ of $r$-uniform hypergraphs, let $$\mathcal{F}\oplus\mathcal{F}'=\bigcup_{\mathcal{T}\in \mathcal{F},\mathcal{T}'\in\mathcal{F}}\mathcal{T}\oplus\mathcal{T}'.$$

\begin{lemma}\label{lemma:gluing}
Let $\mathcal{F}$ and $\mathcal{F}'$ be two families of $r$-uniform hypergraphs. Then 
$$\mbox{ex}(n,\mathcal{F}\oplus\mathcal{F}')\leq 2^{r+1}(\mbox{ex}(n,\mathcal{F})+\mbox{ex}(n,\mathcal{F}')).$$
\end{lemma}
 
The proof of this lemma builds on similar ideas as the ones presented in Section \ref{sect:admissibleedges}. We prepare the proof with a lemma akin to Lemma \ref{lemma:admissible}.
    
Let $\mathcal{F}$ be a family of $r$-uniform hypergraphs, and let $\mathcal{H}$ be an $r$-uniform hypergraph. For $\epsilon,p\in (0,1]$, say that an edge $e\in E(\mathcal{H})$ is \emph{$(\mathcal{F},p,\epsilon)$-rich} if the following holds. Select each vertex of $\mathcal{H}$ independently with probability $p$, and let $U$ be the set of selected vertices. Let $A$ be the event that $\mathcal{H}[U]$ contains a copy of a member of $\mathcal{F}$ which contains the edge $e$. Then $e$ is \emph{$(\mathcal{F},p,\epsilon)$-rich} if $\mathbb{P}(A| e\subset U)> 1-\epsilon$.

\begin{lemma}\label{lemma:rich}
Let $\mathcal{F}$ be a family of $r$-uniform hypergraphs, let $\mathcal{H}$ be an $r$-uniform hypergraph on $n$ vertices, and let $p,\epsilon\in (0,1]$. Then at most $\frac{\mbox{ex}(n,\mathcal{F})}{\epsilon p^{r}}$ edges of $\mathcal{H}$ are not $(\mathcal{F},p,\epsilon)$-rich.
\end{lemma}
    
\begin{proof}
Select each vertex of $\mathcal{H}$ with probability $p$, and let $U$ be the set of selected vertices. For $e\in \mathcal{H}$, let $B_e$ be the event that $e\subset U$ and there exists no copy of a member of $\mathcal{F}$ in $\mathcal{H}[U]$ containing $e$. Also, let $\mathcal{H}'$ be the subhypergraph of $\mathcal{H}$ formed by the not $(\mathcal{F},p,\epsilon)$-rich edges. Note that for each $e\in E(\mathcal{H}')$, we have $\mathbb{P}(B_e)=p^r\mathbb{P}(B_e| e\subset U)\geq \epsilon p^{r}$. 

Let $\mathcal{H}''$ be the subhypergraph of $\mathcal{H}[U]$ formed by those edges $e$ for which $B_e$ happens. By linearity of expectation, we have  $$\mathbb{E}(|E(\mathcal{H}'')|)\geq |E(\mathcal{H}')|\epsilon p^{r},$$
so there exists a choice for $U$ such that $|E(\mathcal{H}'')|\geq |E(\mathcal{H}')|\epsilon p^{r}$. Note that $\mathcal{H}''$ cannot contain any member of a copy of $\mathcal{F}$, so $|E(\mathcal{H}'')|\leq \mbox{ex}(n,\mathcal{F})$, which implies $|E(\mathcal{H}')|\leq \frac{\mbox{ex}(n,\mathcal{F})}{\epsilon p^{r}}.$
\end{proof}    
    
\begin{proof}[Proof of Lemma \ref{lemma:gluing}]
 Let $\mathcal{H}$ be a hypergraph on $n$ vertices with more than $2^{r+1}(\mbox{ex}(n,\mathcal{F})+\mbox{ex}(n,\mathcal{F}'))$ edges. By Lemma \ref{lemma:rich}, the number of edges of $\mathcal{H}$ which are not $(\mathcal{F},\frac{1}{2},\frac{1}{2})$-rich is at most $2^{r+1}\mbox{ex}(n,\mathcal{F})$, and the number of edges of $\mathcal{H}$ which are not $(\mathcal{F}',\frac{1}{2},\frac{1}{2})$-rich is at most $2^{r+1}\mbox{ex}(n,\mathcal{F}')$. Therefore, $\mathcal{H}$ has an edge $e$ which is both $(\mathcal{F},\frac{1}{2},\frac{1}{2})$-rich and $(\mathcal{F}',\frac{1}{2},\frac{1}{2})$-rich.
 
 Color the vertices in $V(\mathcal{H})\setminus e$ red or blue independently with probability $\frac{1}{2}$, and suppose that the vertices of $e$ receive both colors. Then, the probability that there exists a red copy of a member of $\mathcal{F}$ containing $e$ is more than $\frac{1}{2}$. Also, the probability that there exists a blue copy of a member of $\mathcal{F}'$ containing $e$ is more than $\frac{1}{2}$. But then there exists a coloring such that $e$ is contained in both a red copy of some $\mathcal{T}\in\mathcal{F}$, and a blue copy of some $\mathcal{T}'\in\mathcal{F}'$, which means that $\mathcal{H}$ contains $\mathcal{T}\oplus\mathcal{T}'.$
\end{proof}

Let us remark that the statement of Lemma \ref{lemma:gluing} remains true even if we strengthen the definition of $\mathcal{T}\oplus\mathcal{T}'$ by specifying the edges of $\mathcal{T}$ and $\mathcal{T}'$ we wish to glue together. For ease of notation we only state the weaker version, which already serves our purposes.

 Now we are ready to prove the main theorem of this section.

\begin{proof}[Proof of Theorem \ref{thm:torus}]
 Let $g$ be the genus of $\mathcal{S}$. We prove by induction on $g$ that there exists a constant $c(g)$ such that $\mbox{ex}_{hom}(n,\mathcal{S})\leq c(g)n^{5/2}$. The case $g=1$ follows from Theorem \ref{thm:torus}. Suppose that $g>1$. Let $\mathcal{F}$ be the family of triangulations of the surface of genus $g-1$, and let $\mathcal{F}'$ be the family of triangulations of the torus. Then by Lemma \ref{lemma:gluing}, we have
 
$$\mbox{ex}(\mathcal{F}\oplus\mathcal{F}')\leq 16(\mbox{ex}(\mathcal{F},n)+\mbox{ex}(\mathcal{F}',n))\leq 16(c(g-1)+c(1))n^{5/2}.$$

But for each $\mathcal{T}\oplus\mathcal{T}'\in \mathcal{F}\oplus\mathcal{F}'$, after removing the edge $e$ we glued $\mathcal{T}$ and $\mathcal{T}'$ along, we get a triangulation of the surface of genus $g$. Therefore, we can take $c(g)=16(c(g-1)+c(1))$, finishing the proof. 
\end{proof}

\section{Topological cycles}\label{sect:cycles}
In this section, we present the proof of Theorem \ref{thm:cycles}. We will reduce the problem to  finding rainbow cycles in graphs whose vertices are colored with certain colors. We present this problem in the next subsection.

\subsection{Rainbow cycles}

Let $G$ be a graph. If $r$ is a positive integer and $X$ is some base set, call an assignment $f:V(G)\rightarrow X^{(r)}$ an \emph{$r$-set coloring } of $G$, and for $v\in V(G)$, call the $r$-element set $f(v)$ the \emph{color} of $v$. Say that an $r$-set coloring $f$ is \emph{diverse} if for any two distinct vertices $v,w\in V(G)$, if $v$ and $w$ are neighbors, or $v$ and $w$ has a common neighbor, then $f(v)$ and $f(w)$ are disjoint. Finally, say that a subgraph of $G$ is \emph{rainbow}, if any two of its vertices are colored with disjoint sets. The main result of this subsection is the following lemma.

\begin{lemma}\label{lemma:rainbow}
  Let $r,k$ be positive integers, then there exists $c=c(r,k)$ such that the following holds. If $G$ is a graph on $n$ vertices with a diverse $r$-set coloring which contains no rainbow cycle of length at most $2k$, then $G$ has at most $c n^{1+1/k}$ edges.
\end{lemma}

The proof of this lemma follows very closely the ideas presented in a recent paper of Janzer \cite{J20}. The next lemma we prove is a slight modification of Lemma 2.1 in \cite{J20}.

Let $H$ and $G$ be  two graphs. A \emph{homomorphism} (not to confuse with homeomorphism) from $H$ to $G$ is a function $\phi:V(H)\rightarrow V(G)$ such that $\phi(x)\phi(y)\in E(G)$  if $xy\in E(H)$. Also, $\mbox{hom}(H,G)$ denotes the number of homomorphisms from $H$ to $G$. If $G$ is clear from the context, we write simply $\mbox{hom}(H)$ instead of $\mbox{hom}(H,G)$. As usual, $P_{\ell}$ denotes the path of length $\ell$, and $C_\ell$ denotes the cycle of length $\ell$; we also consider $C_2$ as the degenerate cycle of length 2, so $\mbox{hom}(C_2,G)=2|E(G)|$.  If $y,z\in V(G)$, then $\mbox{hom}_{y,z}(P_{\ell})$ denotes the number of walks of length $\ell$ with endpoints $y$ and $z$. Finally, $\Delta(G)$ denotes the maximum degree of $G$.

\begin{lemma}\label{lemma:badhom}
 Let $G$ be a graph with a diverse $r$-set coloring $f$ of the vertices, and let $\ell\geq 2$. Then the number of homomorphisms of $C_{2\ell}$ which are not rainbow is at most $$16\ell(r\ell\Delta(G)\mbox{hom}(C_{2\ell-2})\mbox{hom}(C_{2\ell}))^{1/2}.$$
\end{lemma}

\begin{proof}
  For a positive integer $s$, let $\alpha_s$ be the number of walks of length $\ell-1$ whose endpoints $y,z$ satisfy $2^{s-1}\leq\mbox{hom}_{y,z}(P_{\ell-1})\leq 2^{s}$, and let $\beta_{s}$ be the number of walks of length $\ell$ whose endpoints $y,z$ satisfy $2^{s-1}\leq\mbox{hom}_{y,z}(P_{\ell})\leq 2^{s}$. Then 
  $$\sum_{s\geq 1}\alpha_{s}2^{s-1}<\mbox{hom}(C_{2\ell-2}),$$
  and 
   $$\sum_{s\geq 1}\beta_{s}2^{s-1}<\mbox{hom}(C_{2\ell}).$$
   For positive integers $s$ and $t$, let $\gamma_{s,t}$ denote the number of homomorphic copies $(x_1,\dots,x_{2\ell})$ of $C_{2\ell}$ such that $f(x_1)$ and $f(x_i)$ are not disjoint for some $i\in\{2,\dots,\ell+1\}$, $2^{s-1}\leq\mbox{hom}_{x_1,x_{\ell+2}}(P_{\ell-1})<2^{s}$ and $2^{t-1}\leq \mbox{hom}_{x_2,x_{\ell+2}}(P_\ell)<2^t$.
   
   We can bound $\gamma_{s,t}$ two ways. 
   \begin{enumerate}
       \item $\gamma_{s,t}\leq \alpha_s \Delta(G)2^{t}$. Indeed, there are at most $\alpha_{s}$ ways to choose $x_{\ell+2},x_{\ell+3},\dots,x_{2\ell},x_1$, then, there are at most $\Delta(G)$ ways to choose $x_2$ with $x_1$ already chosen, and there are at most $2^{t}$ ways to choose $x_3,\dots,x_{\ell+1}$ by the inequality $\mbox{hom}_{x_2,x_{\ell+1}}(P_{\ell})<2^t$. 
       
       \item $\gamma_{s,t}\leq \beta_{t}r\ell2^{s}$. Indeed, there are at most $\beta_{t}$ ways to choose the vertices $x_2,\dots,x_{\ell+2}$. Then, there are at most $r\ell$ choices for $x_1$. This is true as $f(x_1)\cap f(x_i)\neq \emptyset$ for some $i\in \{2,\dots,\ell+1\}$, but the neighbors of $x_2$ have disjoint colors by $f$ being diverse, so at most $r$ neighbours of $x_2$ have a color intersecting $f(x_i)$. Finally, there are at most $2^s$ further choices for $x_{\ell+3},\dots,x_{2\ell}$ as $\mbox{hom}_{x_1,x_{\ell+2}}(P_{\ell-1})<2^s$.
       
   \end{enumerate}
   
   The number of homomorphisms of $C_{2\ell}$ which are not rainbow is at most $2k\sum_{s,t\geq 1}\gamma_{s,t}$. Indeed, if $(x_1,\dots,x_{2\ell})$ is a homomorphic copy of $C_{2\ell}$ which is not rainbow, then at least one of its $2k$ cyclic shifts $(x_1',\dots,x_{2\ell}')$ satisfy that $f(x_1')\cap f(x_i')\neq\emptyset$ for some $i\in\{2,\dots,\ell+1\}$. 
   
   Let us bound the sum $\sum_{s,t\geq 1}\gamma_{s,t}$. Let $q$ satisfy $2^{2q}=\frac{r\ell\mbox{hom}(C_{2\ell})}{\Delta(G)\mbox{hom}(C_{2\ell-2})}$. Divide the sum into two parts. Firstly,
   \begin{align*}
    \sum_{s,t:s\leq t-q}\gamma_{s,t}&\leq \sum_{s,t:s\leq t-q}\beta_{t}r\ell2^{s}\leq 2r\ell2^{-q}\sum_{t\geq 1}\beta_{t}2^{t}\\
    &\leq 4r\ell2^{-q}\mbox{hom}(C_{2\ell})=4(r\ell\Delta(G)\mbox{hom}(C_{2\ell})\mbox{hom}(C_{2\ell-2}))^{1/2}.
   \end{align*}
   Secondly,
    \begin{align*}
    \sum_{s,t:s> t-q}\gamma_{s,t}&\leq \sum_{s,t:s> t-q}\alpha_{s}\Delta(G)2^{t}\leq 2\Delta(G)2^{q}\sum_{s\geq 1}\alpha_{s}2^{s}\\
    &\leq 4\Delta(G)2^{q}\mbox{hom}(C_{2\ell-2})=4(rl\Delta(G)\mbox{hom}(C_{2\ell})\mbox{hom}(C_{2\ell-2}))^{1/2}.
   \end{align*}
   Hence, we get $\sum_{s,t\geq 1}\gamma_{s,t}\leq 8(r\ell\Delta(G)\mbox{hom}(C_{2\ell})\mbox{hom}(C_{2\ell-2}))^{1/2}$, finishing the proof.
   \end{proof}
   
   We prepare the proof of Lemma \ref{lemma:rainbow} with two more claims. The fist one is a result of Jiang and Seiver \cite{JS12} stating that not too sparse graphs contain balanced subgraphs.
   
   \begin{claim}(Jiang, Seiver \cite{JS12})\label{claim:JS}
   Let $\alpha>0$, then there exist $c_{1},c_2>0$ such that the following holds. Let $G$ be a graph on $n$ vertices with at least $cn^{1+\alpha}$ edges. Then $G$ contains a subgraph $G'$ on $m$ vertices for some positive integer $m$ such that every degree of $G'$ is between $c_1cm^{\alpha}$ and $c_2cm^{\alpha}$.
   \end{claim}
   
   The final claim we need is that even cycles satisfy Sidorenko's conjecture \cite{S91}, and thus we have a good lower bound on the number of homomorphisms of $C_{2k}$.
   
   \begin{claim}(Sidorenko \cite{S91})\label{claim:Sidorenko}
   Let $G$ be a graph on $n$ vertices. Then 
   $$\mbox{hom}(C_{2k},G)\geq \frac{(2|E(G)|)^{2k}}{n^{2k}}.$$
   \end{claim}
   
   \begin{proof}[Proof of Lemma \ref{lemma:rainbow}]
     We show that $c=c(r,k)=\frac{256k^{3}rc_2}{c_1^2}$ suffices. Let $G$ be a graph on $n$ vertices with at least $cn^{1+1/k}$ edges, and let $f$ be a diverse $r$-set coloring of $G$. By Claim \ref{claim:JS}, $G$ contains a subgraph $G'$ with $m$ vertices such that every degree of $G'$ is between $c_1cm^{1/k}$ and $c_2cm^{1/k}$. In particular, $$ \frac{1}{2}c_1cm^{1+1/k}\leq|E(G')|\leq \frac{1}{2}c_2cm^{1+1/k}.$$
     
     Suppose that $G'$ contains no rainbow copy of $C_{2\ell}$ for $2\leq \ell\leq k$. Then by Lemma \ref{lemma:badhom}, we have $$16\ell(r\ell c_2cm^{1/k}\mbox{hom}(C_{2\ell-2},G')\mbox{hom}(C_{2\ell},G'))^{1/2}\geq \mbox{hom}(C_{2\ell},G'),$$
     or equivalently,
     $$256l^{3}rc_2cm^{1/k}\mbox{hom}(C_{2\ell-2},G')\geq \mbox{hom}(C_{2\ell},G').$$
     But then $$\mbox{hom}(C_{2k},G')<\mbox{hom}(C_2,G')(256 k^{3}rc_2cm^{1/k})^{k-1}<(256k^{3}rc_{2}c)^{k}m^{2}.$$
     On the other hand, by Claim \ref{claim:Sidorenko}, we have
     $$\mbox{hom}(C_{2k},G')\geq (c_1c)^{2k}m^{2}.$$
     Comparing these two inequalities, we get a contradiction by the choice of $c$, so $G'$ contains a rainbow  copy of $C_{2\ell}$ for some $2\leq \ell\leq k$.
   \end{proof}
   
   \subsection{Finding a topological cycle}
   Let us turn to hypergraphs. Let $\mathcal{H}$ be an $r$-partite $r$-uniform hypergraph with vertex classes $A_1,\dots,A_r$. For $i=1,\dots,r$, let $B_{i}$ be the family of $(r-1)$-element sets $X$ in $V(\mathcal{H})$ such that $|X\cap A_j|=1$ for $j\in\{1,\dots,r\}\setminus \{i\}$. Let the \emph{degree of $X$}, denoted by $d(X)=d_{\mathcal{H}}(X)$, be the number of edges of $\mathcal{H}$  containing $X$. Finally, let $C_{i}=C_{i}(\mathcal{H})=\{X\in B_{i}:d(X)>0\}$. 
 
 First, we prove a slightly weaker variant of Claim \ref{claim:JS} for hypergraphs.
 
 \begin{lemma}\label{lemma:balanced}
There exist positive real numbers $c_1=c_1(r)$ and  $c_2=c_2(2)$ such that the following holds. Let $h=h(r)=\binom{r+1}{2}$, and let $|A_1|=\dots=|A_r|=n$. Then there exist positive integers $t_1,\dots,t_r$, and a subhypergraph $\mathcal{H}'$ of $\mathcal{H}$ such that 
\begin{enumerate}
    \item for $i=1,\dots,r$ and $X\in C_{i}(\mathcal{H}')$, we have $t_i\leq d_{\mathcal{H}'}(X)<c_1t_i(\log n)^{h}$,
    \item $\mathcal{H}'$ has at least $c_2|E(\mathcal{H})|(\log n)^{-h}$ edges.
\end{enumerate}
 \end{lemma}
 
 \begin{proof}
  We prove this by induction on $r$. In the case $r=1$, there is nothing to prove, so suppose that $r>1$.
  
  Let $m=|E(\mathcal{H})|$. For $\ell=1,\dots,\log_2 n=:s$, let $D_{\ell}=\{X\in B_{i}: 2^{\ell-1}\leq d(X)<2^{\ell}\}$. Each edge of $\mathcal{H}$ contains exactly one element of one of the sets $D_0,\dots,D_{s}$, so there exists $\ell$ such that at least $\frac{m}{s}$ edges contain an element of $D_{\ell}$. Delete all edges of $\mathcal{H}$ not containing an element of $D_{\ell}$, let the resulting hypergraph be $\mathcal{H}_{0}$.  Let $u_{r}=2^{\ell-1}$ and $p_r=|C_r(\mathcal{H})|$. Then $p_r\leq \frac{|E(\mathcal{H}_{0})|}{u_{r}}\leq \frac{m}{u_r}$.
  
  Now for $v\in A_r$, consider the link graph $\mathcal{H}_v$ of $v$ in $\mathcal{H}_{0}$. Let $h'=h(r-1)$, and $c_1'=c_1(r-1),c_2'=c_2(r-1)$. Then $\mathcal{H}_v$ is an $(r-1)$-partite ${(r-1)}$-uniform hypergraph, so we can apply our induction hypothesis to conclude that there exist positive integers $t_1^{v},\dots,t_{r-1}^{v}$ and a subhypergraph $\mathcal{H}'_v$ of $\mathcal{H}_v$ such that 
  \begin{enumerate}
   \item for $i=1,\dots,r-1$ and $X\in C_{i}(\mathcal{H}'_{v})$, we have $t_i^{v}\leq d_{\mathcal{H}'_{v}}(X)<c_1't_i^{v}(\log n)^{h'}$,
    \item $\mathcal{H}'_{v}$ has at least $c_2'|E(\mathcal{H}_{v})|(\log n)^{-h'}$ edges. 
    \end{enumerate}
    
   For $\overline{\ell}=(\ell_1,\dots,\ell_{r-1})\in [s]^{r-1}$, let $\mathcal{H}_{\overline{\ell}}$ be the hypergraph formed by those edges $X\cup \{v\}\in\mathcal{H}_{0}$, where $X\in \mathcal{H}_v'$ and $2^{\ell_{i}-1}\leq t_{i}^{v}<2^{\ell_{i}}$ for $i=1,\dots,r-1$. We have $$\sum_{v\in A_r}|E(\mathcal{H}'_{v})|\geq c_2'|E(\mathcal{H}_{0})|(\log n)^{-h'}\geq c_2'm(\log n)^{-h'-1},$$ so there exists $\overline{\ell}$ such that $|E(\mathcal{H}_{\overline{\ell}})|\geq c_2'm(\log n)^{-h'-r}=c_2'm(\log n)^{h}$. 
   
   Let $\mathcal{H}_{1}=\mathcal{H}_{\overline{\ell}}$, let $m_1=|E(\mathcal{H}_{1})|$, and for $i=1,\dots,r-1$, let $u_{i}=2^{\ell_i-1}$. In $\mathcal{H}_{1}$, every $X\in C_{i}(\mathcal{H}_{1})$ has degree at most $2c_1'u_{i}(\log n)^{h'}$, and if $i\leq r-1$, then $d(X)\geq u_{i}$. However, if $X\in C_{r}(\mathcal{H}_{1})$, the degree of $X$ might be smaller than $u_r$. For $i\in [r-1]$, let $p_{i}=|C_{r}(\mathcal{H}_{1})|$, then $p_{i}\leq \frac{m_1}{u_i}$. Also, $p_r\leq \frac{c_2'm_1(\log n)^{h}}{u_{r}}$.
   
   For $i\in [r-1]$, let $t_{i}=\frac{1}{2r}u_{i}$, and let $t_{r}=\frac{1}{2rc_2'}u_{r}(\log n)^{-h}$. Now repeat the following procedure. If there exists $i\in [r]$ and $X\in C_{i}(\mathcal{H}_1)$ of such that $d(X)<t_{i}$, then delete all edges from $\mathcal{H}_1$ containing $X$, otherwise stop. Let $\mathcal{H}'$ be the hypergraph we get at the end of the procedure. In total, we deleted at most $\sum_{i=1}^{r}t_ip_{i}$ edges of $\mathcal{H}_1$. But by the choice of $t_{i}$ and the bounds on $p_{i}$, we have $t_{i}p_{i}\leq \frac{m_1}{2r}$ for $i\in [r]$. This means that we deleted at most half of the edges, so $\mathcal{H}'$ is a nonempty hypergraph. Furthermore, we have 
   \begin{enumerate}
       \item for $i=1,\dots,r-1$, if $X\in C_{i}(\mathcal{H}')$, then $t_{i}\leq d_{\mathcal{H}'}(X)<2rc_1't_{i}(\log n)^{h'}$,
       \item if $X\in C_{r}(\mathcal{H}')$, then $t_{r}\leq d_{\mathcal{H}'}(X)<4rc_2'(\log n)^{h}$,
       \item $|E(\mathcal{H}')|\geq \frac{m_1}{2}\geq \frac{c_2}{2}m(\log n)^{-h}$.
   \end{enumerate}
   This shows that setting $c_1=\max\{2rc_1',4rc_2'\}$ and $c_2=\frac{c_2'}{2}$ suffices.
 \end{proof}
 
 Let $\mathcal{H}$ be an $r$-partite $r$-uniform hypergraph  with vertex classes $A_1,\dots,A_r$. If $e=x_1\dots x_r$ and $f=y_1\dots y_r$ are two disjoint edges of $\mathcal{H}$, where $x_i,y_i\in A_i$ for $i\in [r]$, write $e\rightarrow f$ if $x_1\dots x_{i}y_{i+1}\dots y_{r}\in E(\mathcal{H})$ for $i\in [r-1]$. Define the graph $L=L(\mathcal{H})$ such that the vertices of $L$ are the edges of $\mathcal{H}$, and $e,f\in V(L)$ are joined by an edge if $e\rightarrow f$ or $f\rightarrow e$. The graph $L$ is naturally $r$-set colored, where the color of each vertex is itself. 
 
 \begin{lemma}\label{lemma:diverse}
 There exist constants $c_3=c_3(r)$ and $h_1=h_1(r)$ such that the following holds. Let $\mathcal{H}$ be an $r$-partite $r$-uniform hypergraph with vertex classes of size $n$, and suppose that $|E(\mathcal{H})|\geq dn^{r-1}$. Let $L=L(\mathcal{H})$, and consider the natural $r$-set coloring of $L$. Then $L$ contains a subgraph $L'$ with average degree $c_3d(\log n)^{-h_1}$ on which the coloring is diverse.
 \end{lemma}
 
 \begin{proof}
   Let $h,c_1,c_2$ be the constants given by Lemma \ref{lemma:balanced}. We show that $c_3=\frac{c_2}{8rc_1^{r}}$ and $h_1=(r+1)h$ suffices. 
   
   By Lemma \ref{lemma:balanced}, there exist positive integers $t_1,\dots,t_r$ and a subgraph $\mathcal{H}'$ of $\mathcal{H}$ such that 
   \begin{enumerate}
    \item for $i=1,\dots,r$ and $X\in C_{i}(\mathcal{H}')$, we have $t_i\leq d_{\mathcal{H}'}(X)<c_1t_i(\log n)^{h}$,
    \item $\mathcal{H}'$ has at least $c_2dn^{r-1}(\log n)^{-h}$ edges.
 \end{enumerate}
  Let $L_1=L(\mathcal{H}')$, and let $T=t_1\dots t_r$. Note that the degree of each vertex $x_1\dots x_{r}$ of $L_1$ is between $T$ and $2Tc_1^{r}(\log n)^{rh}$. Indeed, given $y_{j+1},\dots,y_r\in V(\mathcal{H}')$ for some $j\geq 1$ such that $x_1\dots x_j y_{j+1}\dots y_r\in E(\mathcal{H})$, there are $$t_{j}\leq d_{\mathcal{H}'}(x_1\dots x_{j-1}y_{j+1}\dots y_r)< c_1t_j(\log n)^{h}$$ choices for the vertex $y_{j}\in V(\mathcal{H}')$ such that $x_1\dots x_{j-1}y_j y_{j+1}\dots y_r\in E(\mathcal{H})$.
  
  Also, as $|E(\mathcal{H}')|\leq |C_{i}(\mathcal{H}')|c_1t_i(\log n)^{h}\leq n^{r-1}c_1t_i(\log n)^{h}$, we get $t_{i}\geq \frac{c_2}{c_1}d(\log n)^{-2h}$ for $i\in[r]$. Define the graph $H$ whose vertices are the edges of $L_1$, and $ef\in E(L_1)$ and $e'f'\in E(L_1)$ are joined by an edge if $e=e'$ and $f\cap f'\neq \emptyset$. Given $e=x_1\dots x_{r}\in V(L_1)$, the number of neighbors of $e$ in $L_1$ containing a given $y_i\in A_i$ is at most $\Delta_i:=\frac{T}{t_i}(\log n)^{(r-1)h}c_1^{r-1}$. Therefore, the degree of every vertex of $H$ is at most 
  $$2\sum_{i=1}^{r}\Delta_i\leq 2\sum_{i=1}^r\frac{T}{t_i}(\log n)^{(r-1)h}c_1^{r-1}\leq \frac{2rT(\log n)^{h(r+1)}c_1^{r}}{c_2d}.$$
   But then $H$ has an independent set of size $$\frac{|V(H)|}{\Delta(H)+1}=\frac{|E(L_1)|}{\Delta(H)+1}>\frac{|V(L_1)|T/2}{4rT(\log n)^{h(r+1)}c_1^{r}/c_2d}=\frac{c_2d|V(L_1)|}{8r(\log n)^{h(r+1)}c_1^{r}}.$$
  This independent set is a subgraph $L'$ of $L_1$ in which the natural $r$-set coloring is diverse, and $L'$ has average degree at least $\frac{c_2d}{8r(\log n)^{h(r+1)}c_1^{r}}=c_3d(\log n)^{-h_1}.$ 
 \end{proof}
 
 Now let us show how the existence of rainbow cycles in $L(\mathcal{H})$ implies the existence of topological cycles in $\mathcal{H}$.
 
 \begin{lemma}\label{lemma:rainbow_to_topological}
  If $\mathcal{H}$ is a 3-partite 3-uniform hypergraph and $L(\mathcal{H})$ contains a rainbow cycle of length $2\ell$, then $\mathcal{H}$ contains a torus-like topological cycle of length at most $6\ell.$
 \end{lemma}
 
 \begin{proof}
    As $\mathcal{H}$ is 3-partite, every topological cycle in $\mathcal{H}$ is torus-like by Claim \ref{claim:3partite}. Therefore, it is enough to show that $\mathcal{H}$ contains a topological cycle.
    
    Let $A_1,A_2,A_3$ be the 3 vertex classes of $\mathcal{H}$ and let $L=L(\mathcal{H})$. For $i\in[2\ell]$, let $f_{i}=x_{i,1}x_{i,2}x_{i,3}\in V(L)=E(\mathcal{H}')$ be the vertices of a rainbow copy of $C_{2\ell}$ in $L$, where $x_{i,j}\in A_{j}$ for $j=1,2,3$. Then $x_{i,j}\neq x_{i',j'}$ for any distinct  $(i,j),(i',j')\in [2\ell]\times[3]$. By the definition of $L$, we have either $f_i\rightarrow f_{i+1}$ or $f_{i+1}\rightarrow f_i$ (indices are meant modulo $2\ell$). Define the edges $f_{i}',f_{i}''$ as follows.
    \begin{enumerate}
        \item  If $f_i\rightarrow f_{i+1}$, let $f_{i}'=x_{i,1}x_{i,2}x_{i+1,3}$ and $f_{i}''=x_{i,1}x_{i+1,2}x_{i+1,3}$.
        \item If $f_{i+1}\rightarrow f_{i}$, let $f_{i}'=x_{i+1,1}x_{i,2}x_{i,3}$ and $f_{i}''=x_{i+1,1}x_{i+1,2}x_{i,3}$.
    \end{enumerate}
     Then $f_i',f_i''$ are also edges of $\mathcal{H}'$. Consider the sequence of edges $f_1,f_1',f_1'',f_2,f_2',f_2'',\dots,f_{2\ell},f_{2\ell}',f_{2\ell}''$. This sequence might not be a topological cycle, but it contains a subsequence which is.  If $i$ is an index such that $f_{i-1}\rightarrow f_{i}$ and $f_{i+1}\rightarrow f_{i}$, or $f_{i}\rightarrow f_{i-1}$ and $f_{i}\rightarrow f_{i+1}$, then remove $f_i$ from the sequence. We show that the resulting subsequence is a proper ordering of the topological cycle~$\mathcal{C}$. 
  
  One way to see this as follows. Let $i\in [2\ell]$.
  \begin{enumerate}
      \item  If $f_{i}\rightarrow f_{i-1}$ and $f_{i}\rightarrow f_{i+1}$, then $f_{i-1}''=x_{i,1}x_{i,2}x_{i-1,3}$ and $f_{i}'=x_{i,1}x_{i,2}x_{i+1,3}$ are consecutive edges in the sequence, and the vertex $x_{i,2}$ does not appear in any other edge of $\mathcal{C}$. Remove $f_{i-1}''$ and $f_{i}'$ from $\mathcal{C}$  and add the 3-element set $g_{i}=x_{i,1}x_{i-1,3}x_{i+1,3}$.
      \item  If $f_{i-1}\rightarrow f_{i}$ and $f_{i+1}\rightarrow f_{i}$, then $f_{i-1}''=x_{i-1,1}x_{i,2}x_{i,3}$ and $f_{i}'=x_{i+1,1}x_{i,2}x_{i,3}$ are consecutive edges in the sequence, and the vertex $x_{i,2}$ does not appear in any other edge of $\mathcal{C}$. Remove $f_{i-1}''$ and $f_{i}'$ from $\mathcal{C}$ and add the 3-element set $g_{i}=x_{i-1,1}x_{i+1,1}x_{i,3}$.
  \end{enumerate}
  After each such replacement, the resulting hypergraph remains homeomorphic to $\mathcal{C}$. But after every such replacement is executed, the resulting hypergraph is a tight cycle, which is a topological cycle. Therefore, $\mathcal{C}$ is a topological cycle as well. See Figure \ref{fig:rainbowcycle} for an illustration.
 \end{proof}
 
 \begin{proof}[Proof of Theorem \ref{thm:cycles}]
  Let $\mathcal{H}$ be a 3-uniform hypergraph with $n$ vertices and $dn^{2}$ edges. Then $\mathcal{H}$ contains a 3-partite subhypergraph with at least $\frac{2}{9}dn^{2}$ edges. By adding isolated vertices, we can assume that each vertex class of $\mathcal{H}'$ has size $n$. 
  
  Let $L=L(\mathcal{H}')$. Let $c_3,h_1$ be the constants given by Lemma \ref{lemma:diverse} in case $r=3$. Then $L$ has a subgraph $L'$ such that the average degree of $L'$ is at least $\frac{2c_3}{9}d(\log n)^{-h_1}$, and the natural 3-set coloring on $L'$ is diverse. 
  
  Let $c=c(3,k)$ be the constant given by Lemma \ref{lemma:rainbow}.  If  the average degree of $L'$ is at least $2c |V(L')|^{1/k}$, then $L'$ contains a rainbow cycle of length at most $2k$. Here, $|V(L')|\leq |E(\mathcal{H}')|\leq dn^{2}$, so if $L'$ contains no rainbow cycle of length at most $2k$, then
  $$2c(dn^{2})^{1/k}>\frac{2c_3}{9}d(\log n)^{-h_1},$$
  which gives $d\leq c'n^{2/(k-1)}(\log n)^{h'}$ with appropriate constants $c',h'>0$.
  
  But if $L'$ contains a a rainbow cycle of length at most $2k$, then $\mathcal{H}$ contains a torus like topological cycle of length at most $6k$ by Lemma \ref{lemma:rainbow_to_topological}.

 \end{proof}

  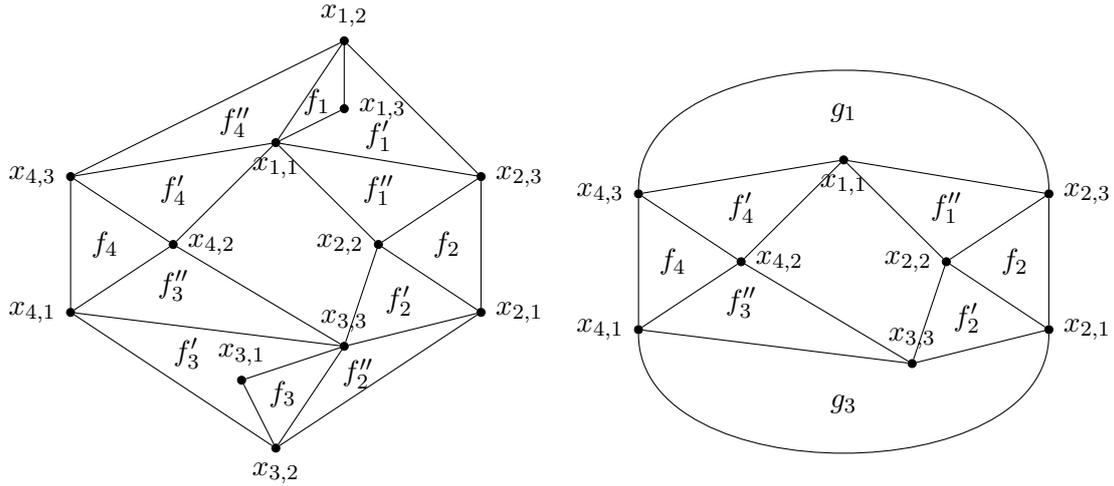
\begin{figure}[!tbp]
	\centering
	\begin{minipage}[b]{0.45\textwidth}
\begin{tikzpicture}[scale=0.9]
 
\node[vertex,label=below:$x_{1,1}$] (x11) at (0,0) {};
\node[vertex,label=above:$x_{1,2}$] (x12) at (1,1.5) {};
\node[vertex,label=right:$x_{1,3}$] (x13) at (1,0.5) {};

\node at (0.6,0.6) {$f_1$};
\node at (1.5,0.1) {$f_1'$};
\node at (1.5,-0.7) {$f_1''$};

\node[vertex,label=right:$x_{2,1}$] (x21) at (3,-2.5) {};
\node[vertex,label=left:$x_{2,2}$] (x22) at (1.5,-1.5) {};
\node[vertex,label=right:$x_{2,3}$] (x23) at (3,-0.5) {};

\node at (2.5,-1.5) {$f_2$};
\node at (1.8,-2.33) {$f_2'$};
\node at (1.2,-3.4) {$f_2''$};

\node[vertex,label=above:$x_{3,1}$] (x31) at (-0.5,-3.5) {};
\node[vertex,label=below:$x_{3,2}$] (x32) at (0,-4.5) {};
\node[vertex,label=above:$x_{3,3}$] (x33) at (1,-3) {};

\node at (0.1,-3.7) {$f_3$};
\node at (-1.3,-3.1) {$f_3'$};
\node at (-1.5,-2.1) {$f_3''$};

\node[vertex,label=left:$x_{4,1}$] (x41) at (-3,-2.5) {};
\node[vertex,label=right:$x_{4,2}$] (x42) at (-1.5,-1.5) {};
\node[vertex,label=left:$x_{4,3}$] (x43) at (-3,-0.5) {};

\node at (-2.5,-1.5) {$f_4$};
\node at (-1.5,-0.7) {$f_4'$};
\node at (-0.6,0.3) {$f_4''$};

\draw  (x13) edge (x12);
\draw  (x11) edge (x12);
\draw  (x13) edge (x11);

\draw  (x21) edge (x23);
\draw  (x23) edge (x22);
\draw  (x22) edge (x21);
\draw  (x31) edge (x32);
\draw  (x32) edge (x33);
\draw  (x31) edge (x33);
\draw  (x43) edge (x42);
\draw  (x42) edge (x41);
\draw  (x43) edge (x41);
\draw  (x12) edge (x23);
\draw  (x11) edge (x23);
\draw  (x11) edge (x22);
\draw  (x21) edge (x33);
\draw  (x21) edge (x32);
\draw  (x22) edge (x33);

\draw  (x42) edge (x11);
\draw  (x12) edge (x43);
\draw  (x43) edge (x11);
\draw  (x42) edge (x33);
\draw  (x41) edge (x32);
\draw  (x33) edge (x41);
 
\end{tikzpicture}
\end{minipage}
\begin{minipage}[b]{0.45\textwidth}
\begin{tikzpicture}[scale=0.9]
\node[vertex,label=below:$x_{1,1}$] (x11) at (0,0) {};

\node at (1.5,-0.7) {$f_1''$};

\node[vertex,label=right:$x_{2,1}$] (x21) at (3,-2.5) {};
\node[vertex,label=left:$x_{2,2}$] (x22) at (1.5,-1.5) {};
\node[vertex,label=right:$x_{2,3}$] (x23) at (3,-0.5) {};

\node at (2.5,-1.5) {$f_2$};
\node at (1.8,-2.33) {$f_2'$};

\node[vertex,label=above:$x_{3,3}$] (x33) at (1,-3) {};

\node at (-1.5,-2.1) {$f_3''$};

\node[vertex,label=left:$x_{4,1}$] (x41) at (-3,-2.5) {};
\node[vertex,label=right:$x_{4,2}$] (x42) at (-1.5,-1.5) {};
\node[vertex,label=left:$x_{4,3}$] (x43) at (-3,-0.5) {};

\node at (-2.5,-1.5) {$f_4$};
\node at (-1.5,-0.7) {$f_4'$};

\draw  (x21) edge (x23);
\draw  (x23) edge (x22);
\draw  (x22) edge (x21);
\draw  (x43) edge (x42);
\draw  (x42) edge (x41);
\draw  (x43) edge (x41);
\draw  (x11) edge (x23);
\draw  (x11) edge (x22);
\draw  (x21) edge (x33);
\draw  (x22) edge (x33);
\draw  (x42) edge (x11);
\draw  (x43) edge (x11);
\draw  (x42) edge (x33);
\draw  (x33) edge (x41);
\path (x23) edge [out=90,in=90] (x43);
\path (x41) edge [out=-90,in=-90] (x21);

\node at (0,0.7) {$g_1$};
\node at (0,-3.6) {$g_3$};
 \end{tikzpicture}
\end{minipage}
\caption{An illustration of a rainbow cycle of length 4, $f_1\rightarrow f_2\rightarrow f_3\leftarrow f_4 \leftarrow f_1$ (left), and the tight cycle we get after the operations (right).}
\label{fig:rainbowcycle}
\end{figure}

 \section{Concluding remarks}
  
  In this paper, we proved that if $\mathcal{S}$ is the triangulation of an orientable surface, then $\mbox{ex}_{hom}(n,\mathcal{S})=O(n^{5/2})$. Does the same bound hold for non-orientable surfaces? In particular, we propose the following conjecture.
  
  \begin{conjecture}
  If $\mathcal{H}$ is a 3-uniform hypergraph with $n$ vertices which does not contain a triangulation of the real projective plane, then $\mathcal{H}$ has $O(n^{5/2})$ edges.
  \end{conjecture}
  
  As every non-orientable closed surface is homeomorphic to several real projective planes glued together, a positive answer to this conjecture together with  Lemma \ref{lemma:gluing} would imply that every closed surface has extremal number $\Theta(n^{5/2})$. More generally, let us highlight the conjecture of Linial \cite{L08,L18} mentioned in the introduction.
  
   \begin{conjecture}
  If $\mathcal{S}$ is a 3-uniform hypergraph, then there exists $c=c(\mathcal{H})>0$ such that $\mbox{ex}_{hom}(n,\mathcal{S})<cn^{5/2}$.
  \end{conjecture}
  
  Finally, it would be interesting to see that which triangulations of the sphere appear in hypergraphs with $O(n^{5/2})$ edges. As we observed, one can find double-pyramids, and one can also glue several double-pyramids together to get a triangulation of the sphere. Are there other type of triangulations one can expect? More precisely, we propose the following conjecture which asks if one can expect to find bounded degree triangulations.
  
  \begin{conjecture}
  There exist two constants $d,c>0$ such that if $\mathcal{H}$ is a hypergraph with $n$ vertices and at least $cn^{5/2}$ edges, then $\mathcal{H}$ contains a triangulation of the sphere in which every vertex has degree at most $d$.
  \end{conjecture}
  
 \section*{Acknowledgements}
 We would like to thank Bhargav Narayanan for bringing this problem to our attention, and Oliv\'er Janzer and Benny Sudakov for valuable discussions. Also, we would like to thank Jacques Verstra\"ete for sharing his ideas of finding the extremal number of the tight cycle of length 24 (which we did not end up using).
 
 Istv\'an Tomon was supported by the SNSF grant 200021\_196965 and MIPT Moscow. Also, all authors acknowledge the support of the grant of the Russian Government N 075-15-2019-1926.


\begin{thebibliography}{99}
\bibitem{BES73}
W. G. Brown, P. Erd\H{o}s, V. T. S\'os,
{\it On the existence of triangulated spheres in 3-graphs, and related problems}, 
Period. Math. Hungar. 3 (1973), 221--228.

\bibitem{C10}
 D. Conlon,
 {\it An Extremal Theorem in the Hypercube},
 Electronic Journal of Combinatorics 17 (2010), R111.

\bibitem{Gao}
Z.C. Gao, {\it The number of rooted triangular maps on a surface}, Journal of Combinatorial Theory, Series B 52 (1991), 236--249. 

\bibitem{J20}
O. Janzer,
{\it Rainbow Tur\'an number of even cycles, repeated patterns and blow-ups of cycles},
arXiv preprint, arXiv:2006.01062.

\bibitem{JS12}
T. Jiang, R. Seiver,
{\it Tur\'an numbers of subdivided graphs},
SIAM J. Discrete Math. 26 (2012), 1238--1255.

\bibitem{K11}
P. Keevash,
{\it Hypergraph Turan problems, Surveys in Combinatorics},
Cambridge University Press (2011), 83--140.

\bibitem{KLNS20}
P. Keevash, J. Long, B. Narayanan, A. Scott,
{\it A universal exponent for homeomorphs},
arXiv preprint, arXiv:2004.02657.

\bibitem{L08}
N. Linial,
{\it What is high-dimensional combinatorics?},
Random–Approx (2008).

\bibitem{L18}
N. Linial,
{\it Challenges of high-dimensional combinatorics,}
Lov\'asz’s Seventieth Birthday Conference (2018).

\bibitem{M67}
W. Mader, 
{\it Homomorphieeigenschaften und mittlere Kantendichte von Graphen},
Math. Ann. 174 (1967), 265--268.

\bibitem{M72}
W. Mader,
{\it Existenz $n$-fach zusammenh\"angender Teilgraphen in Graphen mit gen\"ugend hoher Kantendiech,}
Abh. Math. Sem. Univ. Hamburg 37 (1972), 86--97.

\bibitem{Ma}
W. Mantel,
{\it Problem 28},
Wiskundige Opgaven 10 (1907), 60--61. 

\bibitem{M27}
K. Menger,
 {\it Zur allgemeinen Kurventheorie},
 Fund. Math. 10 (1927), 96--115.
 
 \bibitem{MPS}
D. Mubayi, O. Pikhurko, and B. Sudakov,
{\it Hypergraph Tur\'an Problem:  Some Open Questions},
AIM workshop problem lists, manuscript.
 

\bibitem{S91}
A. Sidorenko,
{\it Inequalities for functionals generated by bipartite graphs},
Diskret Mat. 3 (1991), 50--65 (in Russian), Discrete Math. Applied 2 (1992), 489--504 (in English).

\bibitem{ST20}
B. Sudakov, I. Tomon,
{\it The extremal number of tight cycles},
arXiv preprint, arXiv:2009.00528

\bibitem{Tu} 
P. Tur\'an,
{\it On an extremal problem in graph theory (in Hungarian)},
Mat. Fiz. Lapok  48 (1941),  436--452.


\bibitem{V}
J. Verstra\"ete,
personal communication.

\end{thebibliography}
\end{document}